\documentclass[11pt]{article}
\usepackage{amsmath, amssymb, amsfonts, amsthm, amscd, vmargin}
\usepackage[latin1]{inputenc}
\usepackage[all]{xy}
\usepackage{graphicx,color}
\usepackage{enumerate}
\usepackage{tikz}
\usepackage{ragged2e}
\usepackage{multicol}
\usetikzlibrary{shapes}
\usepackage{dynkin-diagrams}
\tikzset{/Dynkin diagram, Bourbaki arrow}
\usepackage{appendix}
\setmarginsrb{2.4cm}{4.5cm}{2.4cm}{3.3cm}{0cm}{0cm}{0cm}{1.5cm}

\theoremstyle{plain}
\newtheorem{teo}{Theorem}
\newtheorem{lem}[teo]{Lemma}
\newtheorem{prop}[teo]{Proposition}

\theoremstyle{definition}
\newtheorem{defi}{Definition}
\newtheorem{ex}{Example}

\newtheorem{rem}{Remark}

\newtheorem{convention}{Convention}
\newtheorem{notation}{Notation}

\newcommand{\Cbb}{{\mathbb C}}
\newcommand{\Qbb}{{\mathbb Q}}

\newcommand{\Pbb}{{\mathbb P}}

\newcommand{\SL}{\operatorname{SL}}
\begin{document}

\title{Local rigidity of projective smooth horospherical varieties of Picard number two}
\author{Boris Pasquier, L\'ea Villeneuve}

\maketitle

\begin{abstract}
 We study the local rigidity of projective smooth horospherical varieties of rank one and Picard number two. These varieties have been already considered by the second author in a work where their automorphism groups are computed. The results given here are a natural continuation of this work.   
\end{abstract}

\tableofcontents

\section*{Introduction}

 We work over the field of complex numbers. Let $G$ be a connected reductive group. A normal $G$-variety $X$ is said to be horospherical if it contains an open $G$-orbit isomorphic to a torus bundle over a flag $G$-variety. Horospherical varieties are a well-known subfamily of spherical varieties, whose theory and fans that classify them, give useful tools to describe their geometry (see for example \cite{Perrin2018} for an introduction to spherical theory).

The first author studied the geometry of smooth projective horospherical varieties with Picard number $1$ first in \cite{Pasquier2009} and with N.~Perrin in \cite{Pasquier2010}. They (not homogeneous ones) are classified into 5 families (including odd symplectic grassmannians), they are all two-orbit varieties with not reductive automorphism group, and  they are all locally rigid except the $G_2$-variety which deforms into the orthogonal grassmannian of isotropic planes in $\Cbb^7$. These varieties of small Picard number give nice examples to study in several aspects of birational geometry.

Smooth projective horospherical varieties with Picard number $2$ are also classified \cite{Pasquier2020} into two categories and with a very big number of non homogeneous families. The second author computed \cite{Villeneuve2025} the automorphism group of these latter varieties in the case of horospherical varieties of rank $1$ (ie when the torus fiber of the open $G$-orbit is a one-dimensional). The restriction to the one-rank case still gives more than 750 families to study. And we can notice that we get two-orbit varieties with not reductive automorphism groups but also three-orbit varieties with reductive automorphism groups (and not necessarily connected). We can remark that this list of studied varieties, contains toric varieties (by restrincting the action to a subtorus of $G$), for example Hirzebruch varieties (see section~\ref{Hirzebruch}).\\

The aim of this paper is to study the local rigidity of all smooth projective rank one horospherical varieties $\mathbb{X}$ with Picard number $2$. 
We obtain the following result.
\begin{teo} The smooth projective rank one horospherical varieties with Picard number $2$ are all locally rigid, except in ``few '' cases in which $H^1(\mathbb{X},T_{\mathbb{X}})$  is non trivial, listed in Proposition~\ref{cor:list}, and in these cases, there is no obstruction except eventually in one particular case. 
\end{teo}

At the end of the paper, we describe a deformation for some varieties $\mathbb{X}$ from the list of not locally rigid cases. The used method works because the chosen cases are horospherical degenerations of a spherical $G$-variety. They coincide to some cases where $H^1(\mathbb{X},T_\mathbb{X})=\Cbb$, where there is a non zero $G$-fixed point in $H^1(\mathbb{X},T_\mathbb{X})$. But we see that, it does not work in other cases, even for general Hirzebruch varieties, for which families of deformations are however completely described in \cite{Manetti04}. It seams to be a long and difficult work to get explicit descriptions of deformations in all cases of not local rigidity smooth projective rank one horospherical varieties with Picard number $2$.\\

The paper is organized as follows. 

In section~\ref{sec1}, after introducing notation and classical results of representation theory, we explain the strategy to compute $H^1(\mathbb{X},T_\mathbb{X})$ for smooth projective rank one horospherical varieties $\mathbb{X}$ with Picard number $2$. 

In section~\ref{sec2}, we give a characterization to have non trivial $H^1(\mathbb{X},T_\mathbb{X})$ and deduce the list of smooth projective rank one horospherical varieties with Picard number $2$ with non trivial $H^1(\mathbb{X},T_\mathbb{X})$. The details to get the list from the characterization are given in the appendix.

In section~\ref{sec3}, we prove that in almost all cases above there is no obstruction to local deformation. Note that we get Fano or weak-Fano varieties in most cases. 

In section~\ref{sec4}, we give few examples of explicit deformations. 


\section{Notation, context and strategy}\label{sec1}

\subsection{Preliminaries}

Smooth projective rank one horospherical varieties with Picard number $2$ are discribed with several parameters: \begin{itemize}
    \item a connected, reductive algebraic group $G$ equal to a simply connected simple group $G_0$, or to a product of such $G_0$ with at most two other groups (simple or equal to $\mathbb{C}^*$ or $\{1\}$);
    \item a simple root of $G_0$ denoted by $\beta$;
    \item zero to two other simple roots of $G$;
    \item a non negative integer denoted by $a_1$.
\end{itemize}

We can denote by $\mathbb{X}(G,\beta,\alpha_0,\alpha_1,a_1)$ these varieties or just $\mathbb{X}$ if there is no confusion. 

For the varieties studied in this paper, all possible quadruple $(G,\beta,\alpha_0,\alpha_1)$ with $a_1=0$ or $a_1\geq 1$ are given in \cite{Villeneuve2025}. There are given by a roman numeral from (I) to (XVIII) in particular depending of the group, except for the case where $G=G_0\times G_1\times G_2$; followed by an Arabic numeral depending on the choice of simple roots. And we have prime and not prime cases depending of the order in the pair $(\alpha_0,\alpha_1)$ (only when $a_1\geq 1$).


\justifying We now introduce classical notation on the root system associated to the group $G$. First, we fix a maximal torus $T$ and a Borel subgroup $B$ of $G$ containing $T$. Then, we denote by $R$ the set of roots of $(G,T)$, by $R^+$ the set of positive roots and by $S$ the set of simple roots of $(G,B,T)$. We remind that $R^+$ is the subset of $R$ whose elements are the linear combinations of elements of S with non negative coefficients. The choice of $R^+$ gives a partial order $\prec$ on $R^+$ : we have $\alpha\prec\beta\Longleftrightarrow\beta -\alpha$ is a linear combination of simple roots with non negative coefficients. 
Let $\alpha$ be a simple root. We denote by $\varpi_{\alpha}$ the fundamental weight associated to $\alpha$, so that $(\varpi_{\alpha})_{\alpha\in S}$ is the dual basis of $(\alpha^{\vee})_{\alpha\in S}$. We denote by $P_{\alpha}$ the maximal parabolic subgroup containing $B$ such that $\varpi_{\alpha}$ is a character of $P_{\alpha}$. We denote by $\mathfrak{X}(T)$ the lattice of characters of $T$ (or $B$) generated by the $\varpi_{\alpha}$ and $\mathfrak{X}^+(T)\subset\mathfrak{X}(T)$ the monoid of dominant characters.

Let $W$ be the Weyl group of $(G,T)$, it is generated by the reflections $s_\alpha$ associated to the simple roots $\alpha$. When $I\subset S$, we denote by $W_I$ the subgroup of $W$ generated by the reflections associated to the simple roots of $I$. We denote by $P_I$ the intersection of the parabolic subgroups $P_{\alpha}$ with $\alpha\in I$. In particular, $P_{\emptyset}$ = $G$ and $P_S$ = $B$.

For $\lambda\in\mathfrak{X}^+(T)$, we denote by $V(\lambda)$ the irreducible $G$-module of highest weight $\lambda$ and by $\upsilon_{\lambda}$ a highest weight vector of $V(\lambda)$. \newline
We denote by $\rho$ the half sum of positive roots (or equivalently the sum of fundamental weights).

Our main tool is the following classical result.

\color{black}
\begin{teo}[Borel-Weil-Bott] Let $V$ be an irreducible $P$-module of highest weight $\chi$. Denote by $\nu(\chi)$ the vector bundle $G\times^P V$ over $G/P$ and by $w_0^P(\chi)$ the lowest weight of $V$. We either have :

$\ast$ if there exists a root $\alpha$ with $\langle w_0^P(\chi)-\rho,\alpha^\vee\rangle$ = $0$, then, for any $i\geq 0$, $H^i(G/P,\nu(\chi))$ = $0$;

$\ast$ otherwise, there exists an element $w$ of the Weyl group with $\langle w(w_0^P(\chi)-\rho),\alpha^\vee\rangle <0$, for all positive roots $\alpha$. Denote by $\ell(w)$ the length of $w$. Then, we have $H^i(G/P,\nu(\chi))$ = $0$ for $i\neq\ell(w)$ and $H^{\ell(w)}(G/P,\nu(\chi))$ is the $G$-module of highest weight $-w(w_0^P(\chi)-\rho)-\rho$.
\end{teo}

Recall that, in our context, $G$ can equal $G_0$, $G_0\times G_1$ or $G_0\times G_1\times G_2$, where $G_0$ is simple, $G_1$ is either $\{1\}$, $\Cbb^*$ or simple, and $G_2$ is either $\Cbb^*$ or simple (and of type $A_m$). See the list in \cite{Villeneuve2025} for more detail. 

When $G_1$ or $G_2$ are  $\{1\}$ or $\Cbb^*$, there is no root $\alpha_0$ or $\alpha_1$. To get a uniform notation, we introduce the following.

\color{black}
\begin{defi}\label{im}
We say that $\alpha_0$ is an imaginary root if $G_1$ = $\{1\}$ and we denote by $\varpi_{\alpha_0}$ = $0$ by convention. In the same way, we say $\alpha_1$ is an imaginary root if $G_1$ = $\mathbb{C}^*$ or $G_2$ = $\mathbb{C}^*$ and $\varpi_{\alpha_1}$ = id$_{\mathbb{C}^*}$. 
\end{defi}

We denote by $w_0^{\alpha,\beta}$ the biggest element of $W_{S\setminus\{\alpha,\beta\}}$. We know that $w_0^{\alpha,\beta}$ is generated by the reflection $s_{\gamma}$ with $\gamma\in S\setminus\{\alpha,\beta\}$ and since $s_{\gamma}(\varpi_{\alpha})$ = $\varpi_{\alpha}$ and $s_{\gamma}(\varpi_{\beta})$ = $\varpi_{\beta}$ for all $\gamma\in S$, then we obtain $w_0^{\alpha,\beta}(\varpi_{\alpha})$ = $\varpi_{\alpha}$ and $w_0^{\alpha,\beta}(\varpi_{\beta})$ = $\varpi_{\beta}$.

We will also need the following description proved in \cite{Villeneuve2025}.

\begin{lem}\label{coeffs c et d}
With the previous notations, $w_0^{\alpha,\beta}(\alpha^{\vee})$ is the biggest root in $R^{\vee}$ with partial order $\prec$ with coefficient $1$ in $\alpha^{\vee}$ and coefficient $0$ in $\beta^{\vee}$.
\end{lem}

The following notation will be used to characterize the (non) vanishing of $H^1(\mathbb{X},T_\mathbb{X})$.

\begin{notation}\label{not_c_d} For any $i,j\in\{0,1\}$, with $i\neq j$, we denote by $c_{\alpha_i}$ and $d_{\alpha_i}$ the coefficients of $\alpha_i^\vee$ in $w_0^{\alpha_j,\beta}(\alpha_j^{\vee})$ and $w_0^{\alpha_j,\beta}(\beta^{\vee})$ respectively, and by $\alpha_i'$ := $w_0^{\alpha_j,\beta}(\alpha_i)$.
\end{notation}

And we use the natural following convention to deal case with imaginary roots.

\begin{convention} For any $i,j\in\{0,1\}$, with $i\neq j$, if $\alpha_i$ is an imaginary root (see Definition \ref{im}), then $c_{\alpha_i}$ = $c_{\alpha_j}$ = $d_{\alpha_i}$ = $0$.
\end{convention}




The coefficient cannot take any value, and we can compute them quite easily. For example, to use of Theorem~\ref{theo1} in order to get the list of not trivial $H^1(\mathbb{X},T_\mathbb{X})$'s, we can apply the following result.

\begin{lem}\label{lem:valeurs_c}
For any $i\in\{0,1\}$ if $G$ = $G_0$, the coefficient $c_{\alpha_i}$ is always equal to $0$, $1$ or $2$ except (eventually) if $G_0$ is of type $E_7$ and $\beta=\chi_7$ or $G_0$ is of type $E_8$ and $\beta=\chi_7$ or $\chi_8$.

In cases where $G\neq G_0$, we can have similar results. In particular, $c_{\alpha_i}=0$ as soon as $\alpha_0$ and $\alpha_1$ are not roots of the same simple group (including imaginary cases). And in cases (XVIII) (ie when $\alpha_0$ and $\alpha_1$ are both in $G_1$) $c_{\alpha_i}$ is 1 or 2, except in case (8) (ie $G_1$ is of type $G_2$) where $c_{\alpha_i}$ is 1 or 3.

\end{lem}

\begin{proof}
If $G_0$ is of type $A_m$, $B_m$, $C_m$ ot $D_m$, we always have $0\leq c_{\alpha_1}\leq 2$, according to \cite{Bourbaki2007}. 
Denote $S_0$ the set of simple roots of $G_0$.
If $G_0$ is of type $F_4$, $\Gamma_{S_0\setminus\{\beta\}}$ is of type $A_m$, $B_m$, $C_m$ or $D_m$, so we have again $0\leq c_{\alpha_1}\leq 2$ (since $w_0^{\alpha_0,\beta}({\alpha_0}^{\vee})$ is a root of $\Gamma_{S^{\vee}\setminus\{\beta\}}$). 

Similarly if $G_0$ of type $G_2$ or of type $E_6$, then $\Gamma_{S_0\setminus\{\beta\}}$ is of type $A_m$, $B_m$, $C_m$ or $D_m$, and again $0\leq c_{\alpha_1}\leq 2$.

Finally, if $G$ is of type $E_7$ or $E_8$, then $\Gamma_{S\setminus\{\beta\}}$ can be of type $E_6$ or $E_7$, which can give $c_{\alpha_1}>2$. For that, $\beta$ must respectively be $\chi_7$ or, $\chi_7$ or $\chi_8$.

The cases where $G\neq G_0$ are left to the reader.
\end{proof}

\begin{rem}
As in Lemma~\ref{lem:valeurs_c}, we can prove that when $G$ = $G_0$, for any $i\in\{0,1\}$, $d_{\alpha_i}$ is 0, 1 or 2 except (eventually) if $G_0$ is of type $E_7$ and $\alpha_1=\chi_7$ or $G_0$ is of type $E_8$ and $\alpha_1=\chi_7$ or $\chi_8$.  The condition $\lambda>0$ in Theorem~\ref{theo2}, equivalent to $a_1\leq d_{\alpha_0}-2$, cannot be satisfied in a lot of cases where $a_1>0$.

When $G_0$ is of type $G_2$, in the cases (XVII)(1) (resp. (XVII)(1)') we obtain that $d_{\alpha_0}$ (resp. $d_{\alpha_1}$)  equals to $3$.
\end{rem}

\subsection{Strategy}

\justifying Our goal here is to study $H^1(\mathbb{X}, T_{\mathbb{X}})$. If the computation gives $0$, we have, according to Kodaira-Spencer theory, that $\mathbb{X}$ is locally rigid, ie does not admit any local deformation.

Since we are dealing with rank one horospherical varieties, we have three $G$-orbits in $\mathbb{X}$, one open and two closed that we denote by $Y$ and $Z$. Moreover, with previous notation, the closed $G$-orbits are isomorphic to $G/P_{\alpha_0,\beta}$ and $G/P_{\alpha_1,\beta}$, respectively denoted by $Y$ and $Z$.

\color{black} For any smooth rank one horospherical variety, we have the following short exact sequence  \\

$$0\longrightarrow S_X\longrightarrow T_X\longrightarrow N_{\partial X/X}\longrightarrow 0$$ \\

\noindent\justifying with $S_X$ the action sheaf of $X$, ie the subsheaf of $T_X$ made of vector fields tangent to $\partial X$ = $Y\cup Z$, which gives the following long exact sequence  \\

\ \ \ \ $H^0(X,S_X)\longrightarrow H^0(X,T_X)\longrightarrow H^0(\partial X,N_{\partial X/X})\longrightarrow H^1(X,S_X)\longrightarrow H^1(X,T_X)\longrightarrow H^1(\partial X,N_{\partial X/X})\longrightarrow H^2(X,S_X)\longrightarrow H^2(X,T_X)\longrightarrow ...$ \newline\newline

\justifying But for any smooth projective horospherical variety, we know that $\forall i>0$, $H^i(X,S_X)$ = $0$ by \cite[Th.0.4]{Pasquier2010}. Hence, we obtain that 
 \begin{equation}\label{eq1}
    H^1(X,T_X)=H^1(\partial X,N_{\partial X/X}).
 \end{equation}

According to (\ref{eq1}), we just have to compute $H^1(\partial \mathbb{X},N_{\partial \mathbb{X}/\mathbb{X}})$. Indeed, we know that $H^1(\partial\mathbb{X},N_{\partial\mathbb{X}/\mathbb{X}})$ = $H^1(Y,N_{Y/\mathbb{X}})\oplus H^1(Z,N_{Z/\mathbb{X}})$. 

We know give the steps of the method we will use to study the local rigidity, before to illustrate it in an example. 
\begin{enumerate}
    \item We consider the weight $\chi_Y$ or $\chi_Z$ which describes the total space of $N_{Y/\mathbb{X}}$ or $N_{Z/\mathbb{X}}$, which is equal respectively to $\chi_Y=-\varpi_{\alpha_0}+\varpi_{\alpha_1}+a_1\varpi_{\beta}$ and $\chi_Z=\varpi_{\alpha_0}-\varpi_{\alpha_1}-a_1\varpi_{\beta}$, and then we compute $w_0^{\alpha_0,\beta}(\chi_Y)$ and $w_0^{\alpha_1,\beta}(\chi_Z)$.
    
    \item The next step consists in writing $w_0^{\alpha_0,\beta}(\chi_Y)-\rho$ and $w_0^{\alpha_1,\beta}(\chi_Z)-\rho$ as a linear combination of the fundamental weights that we denote by $\varpi_{\chi_j}$ according Bourbaki's numeration. Note that, at that step, we can already say that $H^1(Y,N_{Y/\mathbb{X}})$ = $0$ (resp. $H^1(Z,N_{Z/\mathbb{X}})$ = $0$) as soon as $H^0(Y,N_{Y/\mathbb{X}}) \neq 0$ (resp. $H^0(Z,N_{Z/\mathbb{X}})\neq 0$), ie if all coefficients are negative, or if one coefficient is zero (indeed in that case, all $H^i(\mathbb{X},T_{\mathbb{X}})$ = $0$ for any $i\in\mathbb{Z}$). 
    
    \item For the next step, we suppose that  there exists a simple root $\gamma$ associated to a positive coefficient (ie the coefficient in $\varpi_\gamma$ in $w_0^{\alpha_1,\beta}(\chi_Y)-\rho$, resp. $w_0^{\alpha_1,\beta}(\chi_Z)-\rho$ is positive; or equivalently $\langle w_0^{\alpha_1,\beta}(\chi_Y)-\rho,\gamma^\vee\rangle>0$, resp. $\langle w_0^{\alpha_1,\beta}(\chi_Z)-\rho,\gamma^\vee\rangle>0$). Then we have to find if there exists a simple reflection which sends $w_0^{\alpha_0,\beta}(\chi_Y)-\rho$ or $w_0^{\alpha_1,\beta}(\chi_Z)-\rho$ (depending on the case we consider) to the interior of the antidominant chamber.
\end{enumerate}

\begin{rem}\label{rem:uniquecoeffpositif} 
\begin{itemize}
\item  We only have to check if the simple reflection $s_\gamma$ fits in step 3. Indeed, for any simple root $\chi_j$, we have $s_{\chi_j}(\chi_j^\vee)=-\chi_j^\vee$. Then, if $\chi_j$ is not a simple root $\gamma$ associated to a positive coefficient (as in step~3), the simple reflection $s_{\chi_j}$ does not send  $w_0^{\alpha_0,\beta}(\chi_Y)-\rho$ or $w_0^{\alpha_1,\beta}(\chi_Z)-\rho$ to the antidominant chamber. 
\item We can suppose that there is a unique $\gamma$ associated to a positive coefficient. Indeed, if we consider two distinct roots $\gamma$ and $\gamma'$ associated to a positive coefficient, then there exists $\delta^{\vee}\in R_+^{\vee}$ such that $s_{\gamma}(\delta^{\vee})$ = $\gamma'^{\vee}$ and then $\langle s_{\gamma}(w_0^{\alpha_0,\beta}(\chi_Y)-\rho),\delta^{\vee})\rangle>0$ or $\langle s_{\gamma}(w_0^{\alpha_1,\beta}(\chi_Z)-\rho),\delta^{\vee})\rangle>0$, so that there exists no simple reflection such that all the coefficients become negative.
\end{itemize}
\end{rem}










From now, we first consider the computations with the $G$-orbit $Y$. 
Before to give an example, we are going to give a lemma to simplify the computation.

\begin{lem}\label{calcul}
With the same notations as above, for any $i,j\in\{0,1\}$, with $i\neq j$, when $\alpha_i$ is non imaginary, we have $w_0^{\alpha_j,\beta}(\varpi_{\alpha_i})$ = $c_{\alpha_i}\varpi_{\alpha_j}+d_{\alpha_i}\varpi_{\beta}-\varpi_{\alpha_i'}$. Else, $w_0^{\alpha_j,\beta}(\varpi_{\alpha_i})$ = $\varpi_{\alpha_i}$. \newline
\end{lem}


\begin{proof}
We have : \\

$\langle w_0^{\alpha_j,\beta}(\varpi_{\alpha_i}),{\gamma}^{\vee}\rangle$ = $\langle\varpi_{\alpha_i},w_0^{\alpha_j,\beta}({\gamma}^{\vee})\rangle$ =  $\left\{\begin{array}{ccc}
-1 & \mbox{if} & \gamma = \alpha_i' \\

c_{\alpha_i} & \mbox{if} & \gamma = \alpha_j \\

d_{\alpha_i} & \mbox{if} & \gamma = \beta \\

0 & \mbox{else}. &  \\
\end{array}\right.$ \newline\newline






\end{proof}

\begin{ex} Let us choose case (V)(1) in which $H^0(Y,N_{Y/\mathbb{X}})$ is trivial. We remind that, in this case, we have $G$ = $G_0$ of type $E_6$, $\beta$ = $\chi_1$, $\alpha_0$ = $\chi_2$, $\alpha_1$ = $\chi_3$ and $\chi_Y$ = $-\varpi_{\alpha_0}+\varpi_{\alpha_1}+a_1\varpi_{\beta}$. We also remind that we have $\rho$ = $\sum\limits_{j=1}^m\varpi_{\chi_j}$. Then \\

$w_0^{\alpha_0,\beta}(\chi_Y)-\rho$ = $w_0^{\alpha_0,\beta}(-\varpi_{\alpha_0}+\varpi_{\alpha_1}+a_1\varpi_{\beta})-\sum\limits_{j=1}^6\varpi_{\chi_j}$ 

\setlength{\parindent}{3.1cm} = $-w_0^{\alpha_0,\beta}(\varpi_{\alpha_0})+w_0^{\alpha_0,\beta}(\varpi_{\alpha_1})+a_1w_0^{\alpha_0,\beta}(\varpi_{\beta})-\sum\limits_{j=1}^6\varpi_{\chi_j}$  

= $-\varpi_{\alpha_0}+w_0^{\alpha_0,\beta}(\varpi_{\alpha_1})+a_1\varpi_{\beta}-\sum\limits_{j=1}^6\varpi_{\chi_j}$  $(\ast)$ \newline

\justifying According to Lemma \ref{calcul}, $w_0^{\alpha_0,\beta}(\varpi_{\alpha_1})$ = $-\varpi_{\chi_6}+\varpi_{\chi_2}+\varpi_{\chi_1}$, so we have : \\

$(\ast)$ = $-\varpi_{\chi_2}-\varpi_{\chi_6}+\varpi_{\chi_2}+\varpi_{\chi_1}+a_1\varpi_{\chi_1}-\sum\limits_{j=1}^6\varpi_{\chi_j}$

\setlength{\parindent}{1.2cm} = $a_1\varpi_{\chi_1}-\varpi_{\chi_2}-\varpi_{\chi_3}-\varpi_{\chi_4}-\varpi_{\chi_5}-2\varpi_{\chi_6}$. \newline

\justifying We can observe that the only eventually positive coefficient is $a_1$. Then, the last step consists in computing $\langle s_{\chi_1}(w_0^{\alpha_0,\beta}(\chi_Y)-\rho),{\chi_j}^{\vee}\rangle$, for any $j\in\{1,...,6\}$, by Remark~\ref{rem:uniquecoeffpositif}. \newline


Then, we obtain that $\langle s_{\chi_1}(w_0^{\alpha_0,\beta}(\chi_Y)-\rho),{\chi_j}^{\vee}\rangle$ = $\langle w_0^{\alpha_0,\beta}(\chi_Y)-\rho,s_{\chi_1}({\chi_j}^{\vee})\rangle$

 = $a_1\langle\varpi_{\chi_1},s_{\chi_1}({\chi_j}^{\vee})\rangle-\langle\varpi_{\chi_2},s_{\chi_1}({\chi_j}^{\vee})\rangle-\langle\varpi_{\chi_3},s_{\chi_1}({\chi_j}^{\vee})\rangle-\langle\varpi_{\chi_4},s_{\chi_1}({\chi_j}^{\vee})\rangle-\langle\varpi_{\chi_5},s_{\chi_1}({\chi_j}^{\vee})\rangle-2\langle\varpi_{\chi_6},s_{\chi_1}({\chi_j}^{\vee})\rangle$ \newline

\noindent\justifying and we finally have  

$$\langle s_{\chi_1}(w_0^{\alpha_0,\beta}(\chi_Y)-\rho),{\chi_j}^{\vee}\rangle = \left\{\begin{array}{ccc}
-a_1\leq 0 & \mbox{if} & j = 1 \\

-1<0 & \mbox{if} & j = 2 \\

a_1-1 & \mbox{if} & j = 3 \\

-1<0 & \mbox{if} & j = 4 \\

-1<0 & \mbox{if} & j = 5 \\

-2<0 & \mbox{if} & j = 6 \\
\end{array}\right.$$ 

\justifying If $a_1$ = $0$, we are on a wall of the antidominant chamber and then we have $H^1(\mathbb{X},N_{Y/\mathbb{X}})=0$, and if $a_1>0$, then $a_1-1\geq 0$ so we are not in the antidominant chamber and $H^1(\mathbb{X},N_{Y/\mathbb{X}})=0$. We can do the same study for $Z$, and then conclude if $X$ is locally rigid or not in this case by computing $H^1(\mathbb{X},T_{\mathbb{X}})$.\newline
\end{ex}

\section{Computation of $H^1(\mathbb{X},T_{\mathbb{X}})$}\label{liste}\label{sec2}

\justifying As in the previous example, we first focus on the closed $G$-orbit $Y$, which is isomorphic to $G/P_{\alpha_0,\beta}$; ie we look at $H^1(\mathbb{X},N_{Y/\mathbb{X}})$. We recall that, for any $i,j\in\{0,1\}$, with $i\neq j$, we define integers $c_{\alpha_i}$, $c_{\alpha_j}$, $d_{\alpha_i}$ and  $d_{\alpha_j}$; and recall also that  $\alpha_i'$ denotes $w_0^{\alpha_j,\beta}(\alpha_i)$ (Notation~\ref{not_c_d}). 
We then have the following characterization.

\justifying\begin{teo}\label{theo1}
Let $\mathbb{X}$ be one of the smooth horospherical varieties of Picard number 2, that is not in case (XVIII) (8) or (XVIII) (8)'. The following assertions are equivalent : \\

(1)(a) $\lambda$ = $d_{\alpha_1}+a_1-1>0$ and $2-c_{\alpha_1}>0$, and

\setlength{\parindent}{1.2cm} (b) $\beta$ is linked to at most $\alpha_0$ and $\alpha_1'$, and

\setlength{\parindent}{1.2cm} (c) $\left\lbrace\begin{array}{c}
A\lambda+(c_{\alpha_1}-2)<0 \\

B\lambda-2<0 \\
\end{array}\right.$ \\

\justifying where $A$ = $-\langle{\beta},\alpha_0^{\vee}\rangle$ or $0$ if $\alpha_0$ is imaginary, and $B$ = $-\langle{\beta},\alpha_1'^{\vee}\rangle$ or $0$ if $\alpha_1$ is imaginary, \\

(2) $H^1(Y,N_{Y/\mathbb{X}})$ is not trivial. \\
\end{teo}

\begin{rem}\begin{itemize}
    \item We can also write condition (c) into three distinct cases: \\
$\ast$ $\lambda\geq 2$, $c_{\alpha_1}\in\{0,1\}$ and $A=B=0$;\\
$\ast$ $\lambda$ = $1$, $c_{\alpha_1}$ = $0$, $A$ and $B$ in $\{0,1\}$;\\
$\ast$ $\lambda$ = $1$, $c_{\alpha_1}$ = $1$, $A$ = $0$ and $B$ in $\{0,1\}$.
\item In imaginary cases, we can rewrite Theorem~\ref{theo1} by replacing 
(1)(a) by $\lambda>0$ and (1)(c) by $B\lambda-2<0$, $A\lambda-2<0$ or an empty condition respectively in the three types of imaginary cases considered in the proof of the theorem.
\item If $H^1(Y,N_{Y/\mathbb{X}})$ is not trivial and  $\lambda\geq 2$, then $A=B=0$ (including imaginary cases). Then, with condition (1)(b), it means that $\beta$ is orthogonal to any other simple roots and hence $G_0=\operatorname{SL}_2$. In all other cases, $\lambda$ has to be 1 and then only one value of $a_1$ can give a non trivial $H^1(Y,N_{Y/\mathbb{X}})$. 

Note also that if $G_0=\operatorname{SL}_2$ then the conditions in (1) are all satisfied if and only if  $a_1\geq 2$ (because $\lambda=a_1-1$, $c_{\alpha_1}=0$ and $A=B=0$).
\end{itemize}
 \end{rem}

\begin{proof} \justifying Applying step 2 of our strategy, in the case where $\alpha_0$ and $\alpha_1$ are simple roots (not imaginary roots) we have : \\

$w_0^{\alpha_0,\beta}(\chi_Y)-\rho$ = $-\varpi_{\alpha_0}+c_{\alpha_1}\varpi_{\alpha_0}+d_{\alpha_1}\varpi_{\beta}-\varpi_{\alpha_1'}+a_1\varpi_{\beta}-\sum\limits_{l=1}^m\varpi_{\chi_l}$ 

\setlength{\parindent}{3.1cm} = $(c_{\alpha_1}-2)\varpi_{\alpha_0}-2\varpi_{\alpha_1'}+(d_{\alpha_1}+a_1-1)\varpi_{\beta}-\sum\limits_{\substack{l=1 \\ \chi_l\neq\beta \\ \chi_l\neq\alpha_0 \\ \chi_l\neq\alpha_1'}}^m\varpi_{\chi_l}$. \newline

\justifying According to Remark~\ref{rem:uniquecoeffpositif}, we are reduced to the two following possibilities : \\

\setlength{\parindent}{1cm} $\bullet$ either $c_{\alpha_1}-2<0$ and $d_{\alpha_1}+a_1-1>0$,

\setlength{\parindent}{1cm} $\bullet$ or $c_{\alpha_1}-2>0$ and $d_{\alpha_1}+a_1-1<0$. \newline

\justifying But, by Lemma~\ref{lem:valeurs_c}, $c_{\alpha_1}-2>0$ can only occur in types $E_7$ and $E_8$ with $\beta$ equal to $\chi_7$ or $\chi_8$. In all the cases studied, we do not have this situation. \newline

Hence $H^1(\mathbb{X},T_{\mathbb{X}})$ not trivial implies that $c_{\alpha_1}-2<0$ and $d_{\alpha_1}+a_1-1>0$. In particular, $\beta$ is the simple root associated to the positive coefficient.

We now apply step 3 of our strategy: we have to check if $s_\beta(w_0^{\alpha_0,\beta}(\chi_Y)-\rho)$ is in the interior of the antidominant chamber or not.

We compute, for any simple root $\chi_j$, \\

$$(\ast_1):=\langle s_{\beta}(w_0^{\alpha_0,\beta}(\chi_Y)-\rho),{\chi_j}^{\vee}\rangle = \langle w_0^{\alpha_0,\beta}(\chi_Y)-\rho,s_{\beta}({\chi_j}^{\vee})\rangle$$


$$= (c_{\alpha_1}-2)\langle\varpi_{\alpha_0},s_{\beta}(\overset{\vee}{\chi_j})\rangle-2\langle\varpi_{\alpha_1'},s_{\beta}(\overset{\vee}{\chi_j})\rangle+(d_{\alpha_1}+a_1-1)\langle\varpi_{\beta},s_{\beta}(\overset{\vee}{\chi_j})\rangle-\sum\limits_{\substack{l=1 \\ \chi_l\neq\beta \\ \chi_l\neq\alpha_0 \\ \chi_l\neq\alpha_1'}}^m\langle\varpi_{\chi_l},s_{\beta}({\chi_j}^{\vee})\rangle$$.  \\

We can distinguish three distinct cases : \\

$\bullet$ (i) $\chi_j$ = $\beta$ 

$\bullet$ (ii) $\chi_j$ and $\beta$ are orthogonal (ie they are not linked on the Dynkin diagram, or equivalently $\langle\beta,\chi_j^\vee\rangle=0$).

$\bullet$ (iii) $\chi_j$ and $\beta$ are linked on the Dynkin diagram (ie $\langle\beta,\chi_j^\vee\rangle\leq 1$). \\

In the case (i), we have : \\

$(\ast_1)$  = -$(d_{\alpha_1}+a_1-1)$ \newline

\justifying In the case (ii), we have $s_\beta(\chi_j^\vee)=\chi_j^\vee$ : \\

$(\ast_1)$ is $c_{\alpha_1}-2$ if $\chi_j=\alpha_0$, $-2$ if  $\chi_j=\alpha_1'$, and $-1$ else. \newline

\justifying Finally, in the case (iii), we have $s_\beta(\chi_j^\vee)=\chi_j^\vee-\langle\beta,\chi_j^\vee\rangle\beta^\vee$ so that: \\

$$(\ast_1)  = \left\{\begin{array}{ccc}
(c_{\alpha_1}-2)-\langle\beta,\alpha_0^{\vee}\rangle(d_{\alpha_1}+a_1-1) & \mbox{if} & \chi_j = \alpha_0 \\
-2-\langle\beta,\alpha_1'^{\vee}\rangle(d_{\alpha_1}+a_1-1) & \mbox{if} & \chi_j = \alpha_1' \\
-1-\langle\beta,\chi_j^{\vee}\rangle(d_{\alpha_1}+a_1-1) & \mbox{else} &
\end{array}\right.$$ \newline

In particular, if there exists $\chi_j$, distinct from $\alpha_0$ and $\alpha_1'$, with $\langle\beta,\chi_j^\vee\rangle\leq 1$, then we get a non-negative coefficient $-1-\langle\beta,\chi_j^\vee\rangle(d_{\alpha_1}+a_1-1)$ (with the hypothesis that $d_{\alpha_1}+a_1-1>0$).

When at least one of $\alpha_0$ or $\alpha_1$ is imaginary, we can apply Borel-Weil-Bott Theorem to the semi-simple part of $G$ and then replace $\chi_Y$ the character obtained from $\chi_Y$ by deleting the terms $\varpi_{\alpha_i}$ for which $\alpha_i$ is imaginary. Then, we do the same computations, without the terms in $\varpi_{\alpha_0}$ or, $\varpi_{\alpha_1}$ and $\varpi_{\alpha_1}'$ (depending of the cases).
Let us precise the main steps in the three possible cases.

\begin{itemize}
    \item If $\alpha_0$ is imaginary but not $\alpha_1$, we replace $\chi_Y$ by $$-2\varpi_{\alpha_1'}+(d_{\alpha_1}+a_1-1)\varpi_\beta-\sum\limits_{\substack{l=1 \\ \chi_l\neq\beta \\ \chi_l\neq\alpha_1'}}^m\varpi_{\chi_l}.$$ 
    In particular, to get not trivial $H^1(Y,N_{Y/\mathbb{X}})$ we need that $\lambda:=d_{\alpha_1}+a_1-1>0$. (Here, $d_{\alpha_1}$ is the coefficient of $\alpha_1$ in $w_0^\beta(\beta^\vee)$.) \\
    Moreover $(\ast_1)$  becomes $$ -2\langle\varpi_{\alpha_1'},s_{\beta}(\overset{\vee}{\chi_j})\rangle+(d_{\alpha_1}+a_1-1)\langle\varpi_{\beta},s_{\beta}(\overset{\vee}{\chi_j})\rangle-\sum\limits_{\substack{l=1 \\ \chi_l\neq\beta  \\ \chi_l\neq\alpha_1'}}^m\langle\varpi_{\chi_l},s_{\beta}({\chi_j}^{\vee})\rangle$$ 
    which is $-\lambda$ in case (i), $-2$ if $\chi_j=\alpha_1'$ and $-1$ else in case (ii), and $-2+B\lambda$ if $\chi_j=\alpha_1'$ or $-1-\langle \beta,\chi_j^\vee\rangle\lambda$ else in case (iii).
    
    \item
    If $\alpha_1$ is imaginary but not $\alpha_0$, we replace $\chi_Y$ by $$-2\varpi_{\alpha_0}+(a_1-1)\varpi_\beta-\sum\limits_{\substack{l=1 \\ \chi_l\neq\beta \\ \chi_l\neq\alpha_0}}^m\varpi_{\chi_l}.$$ In particular, to get not trivial $H^1(Y,N_{Y/\mathbb{X}})$ we need that $\lambda:=a_1-1>0$.\\
    Moreover $(\ast_1)$  becomes
    $$-2\langle\varpi_{\alpha_0},s_{\beta}(\overset{\vee}{\chi_j})\rangle+(a_1-1)\langle\varpi_{\beta},s_{\beta}(\overset{\vee}{\chi_j})\rangle-\sum\limits_{\substack{l=1 \\ \chi_l\neq\beta \\ \chi_l\neq\alpha_0}}^m\langle\varpi_{\chi_l},s_{\beta}({\chi_j}^{\vee})\rangle$$
    which is $-\lambda$ in case (i), $-2$ if $\chi_j=\alpha_0$ and $-1$ else in case (ii), and $-2+A\lambda$ if $\chi_j=\alpha_0$ or $-1-\langle \beta,\chi_j^\vee\rangle\lambda$ else in case (iii).
    \item
    If both $\alpha_0$ and $\alpha_1$ are imaginary, we replace $\chi_Y$ by $$(a_1-1)\varpi_\beta-\sum\limits_{\substack{l=1 \\ \chi_l\neq\beta \\ \chi_l\neq\alpha_1'}}^m\varpi_{\chi_l}.$$ In particular, to get not trivial $H^1(Y,N_{Y/\mathbb{X}})$ we need that $\lambda:=a_1-1>0$.\\
    Moreover $(\ast_1)$  becomes
    $$ (a_1-1)\langle\varpi_{\beta},s_{\beta}(\overset{\vee}{\chi_j})\rangle-\sum\limits_{\substack{l=1 \\ \chi_l\neq\beta \\ \chi_l\neq\alpha_0 \\ \chi_l\neq\alpha_1'}}^m\langle\varpi_{\chi_l},s_{\beta}({\chi_j}^{\vee})\rangle$$
    which is $-(a_1-1)=-\lambda$ in case (i), $-1$ in case (ii), and $-1-\langle \beta,\chi_j^\vee\rangle\lambda$ in case (iii).
\end{itemize}\end{proof}

If we now deal with the closed $G$-orbit $Z$, which is isomorphic to $G/P_{\alpha_1,\beta}$, then we have the following.

\justifying\begin{teo}\label{theo2}\label{th:equivalenceZ}
Let $\mathbb{X}$ be one of the smooth horospherical varieties of Picard number 2, that is not in case (XVIII) (8) or (XVIII) (8)'. The following assertions are equivalent : \\

(1)(a) $\lambda$ = $d_{\alpha_0}-a_1-1>0$ and $2-c_{\alpha_0}>0$, and

\setlength{\parindent}{1.2cm} (b) $\beta$ is linked to at most $\alpha_1$ and $\alpha_0'$, and

\setlength{\parindent}{1.2cm} (c) $\left\lbrace\begin{array}{c}
A\lambda+(c_{\alpha_0}-2)<0 \\

B\lambda-2<0 \\
\end{array}\right.$ \\

\justifying where $A$ = $-\langle{\beta},{\alpha_1}^{\vee}\rangle$ or $0$ if $\alpha_1$ is imaginary, and $B$ = $-\langle{\beta},{\alpha_0'}^{\vee}\rangle$ or $0$ if $\alpha_0$ is imaginary, \\

(2) $H^1(Z,N_{Z/\mathbb{X}})$ is not trivial. \\
\end{teo}

\begin{proof} Now we deal with the computation of $H^1(Z,N_{Z/\mathbb{X}})$, which begins with the computation of $w_0^{\alpha_1,\beta}(\chi_Z)$. In the case where $\alpha_0$ and $\alpha_1$ are simple roots (not imaginary roots) we have: \\

$$w_0^{\alpha_1,\beta}(\chi_Z)-\rho =-2\varpi_{\alpha_0'}+ (c_{\alpha_0}-2)\varpi_{\alpha_1}+(d_{\alpha_0}-a_1-1)\varpi_{\beta}-\sum\limits_{\substack{l=1 \\ \chi_l\neq\beta \\ \chi_l\neq\alpha_0 \\ \chi_l\neq\alpha_1'}}^m\varpi_{\chi_l}. $$

We obtain the same result as for $Y$, by exchanging indices 0 and 1 in all $\alpha$'s and by replacing $a_1$ by $-a_1$. 

And we can check that it works similarly for imaginary cases.\end{proof}

\color{black}
After studying the list of the cases, we can notice that actually there are few cases that verify the conditions of Theorem \ref{theo1} and Theorem \ref{theo2}, some for $a_1$ = $0$ and other for some positive values of $a_1$. \\

\begin{prop}\label{cor:list}
When $a_1$ = $0$, $H^1(\mathbb{X},T_{\mathbb{X}})$ is not trivial in and only in the following cases :

\begin{multicols}{2}
\setlength{\parindent}{1cm} $\bullet$ (II)( 3) 

\setlength{\parindent}{1cm} $\bullet$ (VIII) (6) 

\setlength{\parindent}{1cm} $\bullet$ (X) (15) 

\setlength{\parindent}{1cm} $\bullet$ (XI) (4) with $m$ = $3$ 

\setlength{\parindent}{1cm} $\bullet$ (XI) (7) with $k$ = $m-1$ 

\setlength{\parindent}{1cm} $\bullet$ (XVI) (9).

\end{multicols}

When $a_1>0$, $H^1(\mathbb{X},T_{\mathbb{X}})$ is not trivial in and only in the following cases :

\begin{multicols}{2}
\setlength{\parindent}{1cm} $\bullet$ (I) (3) for $a_1$ = $1$ 

\setlength{\parindent}{1cm} $\bullet$ (III) (5)' for $a_1$ = $1$ 

\setlength{\parindent}{1cm} $\bullet$ (IV) (2)'for $a_1$ = $1$

\setlength{\parindent}{1cm} $\bullet$ (IV) (5) with $k$ = $1$ for $a_1$ = $1$

\setlength{\parindent}{1cm} $\bullet$ (VIII) (5) for $a_1$ = $1$

\setlength{\parindent}{1cm} $\bullet$ (IX) (1) for $a_1$ = $2$

\setlength{\parindent}{1cm} $\bullet$ (IX) (2)' for $a_1$ = $1$

\setlength{\parindent}{1cm} $\bullet$ (IX) (5)' for $a_1$ = $1$

\setlength{\parindent}{1cm} $\bullet$ (IX) (7) with $k$ = $2$ for $a_1$ = $1$

\setlength{\parindent}{1cm} $\bullet$ (IX) (12) for $a_1$ = $1$

\setlength{\parindent}{1cm} $\bullet$ (IX) (17)' for $a_1$ = $1$

\setlength{\parindent}{1cm} $\bullet$ (X) (16) for $a_1$ = $2$

\setlength{\parindent}{1cm} $\bullet$ (XI) (1) for $a_1$ = $2$

\setlength{\parindent}{1cm} $\bullet$ (XI) (4)' for $a_1$ = $1$

\setlength{\parindent}{1cm} $\bullet$ (XVI) (11) for $a_1$ = $1$

\setlength{\parindent}{1cm} $\bullet$ (XVII) (1) for $a_1$ = $1$ (Z)

\setlength{\parindent}{1cm} $\bullet$ (XVII) (1) for $a_1$ = $2$ (Y)

\setlength{\parindent}{1cm} $\bullet$ (XVIII) (1) to (XVIII) (8), with $G_0$ of type $A_1$ and $a_1 \geq 2$

\setlength{\parindent}{1cm} $\bullet$ (XVIII) (8) for any $G_0$ (Z)

\setlength{\parindent}{1cm} $\bullet$ $G=G_0\times G_1\times G_2$ with $G_0$ of type $A_1$ and $a_1 \geq 2$.
\end{multicols}
\end{prop}

\begin{proof}
We apply Theorem \ref{theo1} and Theorem \ref{theo2}, see the appendix for the details and the way to deal with a lot of cases with as few work as possible.

There are two cases left : (XVIII) (8) and (XVIII) (8)' (with $a_1>0$).

In case (XVIII) (8) $$w_0^{\alpha_0,\beta}(\chi_Y)-\rho=-\varpi_{\alpha_0}-2\varpi_{\alpha_1}+(a_1-1)\varpi_\beta-\sum\limits_{\substack{\chi_l\in S \\ \chi_l\neq\beta \\ \chi_l\neq\alpha_0 \\ \chi_l\neq\alpha_1}}\varpi_{\chi_l}$$
and, as for other cases (XVIII), we have non zero $H^1$ when and only when $G_0$ of type $A_1$ and $a_1\geq 2$. And  $$w_0^{\alpha_1,\beta}(\chi_Z)-\rho=-2 \varpi_{\alpha_0}+\varpi_{\alpha_1}-(a_1+1)\varpi_\beta-\sum\limits_{\substack{\chi_l\in S \\ \chi_l\neq\beta \\ \chi_l\neq\alpha_0 \\ \chi_l\neq\alpha_1}}\varpi_{\chi_l}$$ and applying $s_{\alpha_1}$ we get a weight in the strict antidominant chamber for any $a_1\geq 1$, and then a non trivial $H^1$ (for any type of $G_0$). Indeed, $s_{\alpha_1}(-2 \varpi_{\alpha_0}+\varpi_{\alpha_1})=- \varpi_{\alpha_0}-\varpi_{\alpha_1}$, so that $H^1(Z,N_{Z/\mathbb{X}})=V(a_1\varpi_\beta)$.

In case (XVIII) (8)' $$w_0^{\alpha_0,\beta}(\chi_Y)-\rho=\varpi_{\alpha_0}-2\varpi_{\alpha_1}+(a_1-1)\varpi_\beta-\sum\limits_{\substack{\chi_l\in S \\ \chi_l\neq\beta \\ \chi_l\neq\alpha_0 \\ \chi_l\neq\alpha_1}}\varpi_{\chi_l}$$ and we cannot have non trivial $H^1$ for $a_1>0$. And $$w_0^{\alpha_1,\beta}(\chi_Z)-\rho=-2 \varpi_{\alpha_0}-\varpi_{\alpha_1}-(a_1+1)\varpi_\beta-\sum\limits_{\substack{\chi_l\in S \\ \chi_l\neq\beta \\ \chi_l\neq\alpha_0 \\ \chi_l\neq\alpha_1}}\varpi_{\chi_l}$$ and we get non trivial $H^0$ but trivial $H^1$. \\

\end{proof}

\begin{rem}
In case (XVIII) (8), we can have not trivial $H^1$ and $\beta$ linked to simple roots different from $\alpha$ or $\alpha'$ (when $G_0$ can be of any type). In particular, that case does not satisfy Theorems~\ref{theo1} and~\ref{theo2}.

Also note that in case (XVIII) (8) with $G_0$ of type $A_1$ and $a_1\geq 2$, we have some non trivial $H^1$ from both orbits $Y$ and $Z$. They are the only cases where it occurs.
\end{rem}

\section{Obstructions to local deformations?}\label{sec3}

In the cases where $H^1(\mathbb{X},T_\mathbb{X})$ is not trivial, we prove the following result, which says that there are almost always no obstruction to local deformations.

\begin{teo}\label{th:obstruction}
Let $X$ be a smooth rank one horospherical variety of Picard number two such that $H^1(\mathbb{X},T_\mathbb{X})$ is not trivial. Then, there is no obstruction to local deformations except eventually in the case case (XVIII) (8), with $G_0=\operatorname{SL}_3$ and $a_1\geq 3$.
\end{teo}

The strategy of the proof of Theorem~\ref{th:obstruction} is the following : we first select Fano or weak Fano cases, then we consider two-orbit cases, and finally  we do a case by case study for the few cases left (only one case really needs work).\\

It is well-known that deformations of Fano varieties (ie with ample anticanonical divisor) are unobstructed (immediate consequence of Kodaira-Nakano vanishing) and it was recently proved that deformations of weak Fano varieties (ie with nef and big anticanonical divisor) are also unobstructed \cite[Th.1.1]{Sano2013}. In several cases, we can observe that $\mathbb{X}$ is Fano or weak Fano, and then deduce there is no obstruction to local deformations.

For $\Qbb$-factotial spherical varieties (and then horospherical varieties), anticanonical divisors are known \cite{brioncurves}, and there exists an ample criterion, as a globally generated one \cite{briondiv}. Moreover, for these varieties, nef and globally generated is equivalent. Also, any effective divisor on a spherical is linearly equivalent (or numerically equivalent) to an effective $B$-stable divisor, the effective cone is generated by the classes of the prime $B$-stable divisors (which are finitely many); and an anticanonical divisor of spherical varieties is a linear combinaison of the prime $B$-stable divisors with positive coefficients. Hence the anitcanonical divisor of a $\Qbb$-factorial spherical variety is always big.

We then deduce the following result as a particular case of \cite{Pasquier2006}. For any $\gamma\in S$, we define $R_\mathbb{X}^+$ the set of positive root generated by simple roots in $S\backslash\{\beta,\alpha_0,\alpha_1\}$ and  $$a_\gamma=2-\langle \sum_{\delta\in R^+_\mathbb{X}}\delta,\gamma^\vee\rangle (\geq 2).$$ And we set $a_\gamma=1$ if $\gamma$ is imaginary.\\

\begin{prop}
Let $\mathbb{X}=\mathbb{X}(G,\beta,\alpha_0,\alpha_1,a_1)$ be a smooth projective rank one horospherical variety with Picard number two. 

Then $\mathbb{X}$ is Fano (resp. weak Fano) if and only if $a_1a_{\alpha_1}<a_\beta$ (resp. $a_1a_{\alpha_1}\leq a_\beta$).
\end{prop}

As a direct consequence, we obtain that when $a_1=0$, all $\mathbb{X}$ are Fano varieties.

We can also check that, in all cases with not trivial $H^1(\mathbb{X},T_\mathbb{X})$ where $a_1=1$, except case (XVIII) (8) (with $G_0$ not of type $A_1$), we get $a_\beta=a_{\alpha_1}$ and then all these $\mathbb{X}$ are weak Fano. In fact, we notice that the computation is done with $\delta$'s in the same positive root system of type $A_m$ (or empty), in both cases. For case (XVIII) (8) with $G_0$ not of type $A_1$, we have $a_\beta>a_{\alpha_1}=2$, so that for $a_1=1$, $\mathbb{X}$ is Fano.

In all other cases different from case (XVIII) (8) with $G_0$ not of type $A_1$ or cases where $\alpha_1$ is imaginary and $a_1=2$, $\mathbb{X}$ is not weak Fano. \\

To deal with the remaining cases, we find other arguments to prove that $H^2(\mathbb{X},T_\mathbb{X})$ is trivial, which implies that there is no obstruction to deformations.

 First, we notice that if $\mathbb{X}$ is a two-orbit variety, and $H^1(\mathbb{X},T_\mathbb{X})$ is not trivial, then $H^2(\mathbb{X},T_\mathbb{X})$ has to be trivial. Indeed, if $\mathbb{X}$ is a  two-orbit variety then the normal bundle over one of the two closed $G$-orbit has not trivial $H^0$, and if moreover $H^1(\mathbb{X},T_\mathbb{X})$ is not trivial then the normal bundle over the other  closed $G$-orbit is non trivial $H^0$. But, with Borel-Weil-Bott Theorem, these normal bundles have at most one not trivial $H^i$, so that $H^2(\mathbb{X},T_\mathbb{X})$, which is the direct sum of the $H^2$'s of the two normal bundle, is trivial.\\

There are now few cases left (not weak-Fano, three-orbit varieties): (XVII) (1) with $a_1=2$, (XVIII) (3), (4), (5), (7) and (8) for any  $a_1\geq 2$, for which $H^1(Y,N_{Y/\mathbb{X}})\neq 0$; and case (XVIII) (8) with any $G_0$, for which $H^1(Z,N_{Z/\mathbb{X}})\neq 0$. For the first cases, we conclude quite easily by looking at the normal bundle $N_{Z/\mathbb{X}}$.

In Case (XVII) (1) with $a_1=2$, $w_0^{\alpha_1,\beta}(\chi_Z)-\rho$ has coefficient $\lambda=d_{\alpha_0}-a_1-1=0$ in $\varpi_\beta$ (because $d_{\alpha_0}=3$) so we have trivial $H^i(Y,N_{Y/\mathbb{X}})$'s.

In Cases (XVIII) (3), (4), (5), (7) and (8) with $G_0$ of type $A_1$, $w_0^{\alpha_1,\beta}(\chi_Z)-\rho$ has coefficient $c_{\alpha_0}-2=0$ in $\varpi_{\alpha_1}$ so we have trivial $H^i(Y,N_{Y/\mathbb{X}})$'s.\\

The rest of the section is devoted to the case (XVIII) (8) with any $G_0$ for any $a_1\geq 2$, for which we have to study $H^2(Y,N_{Y/\mathbb{X}})$. 

We already know that if $G_0$ is of type $A_1$ then it vanishes because $H^1(Y,N_{Y/\mathbb{X}})$ is not trivial. Recall that $$w_0^{\alpha_0,\beta}(\chi_Y)-\rho=-\varpi_{\alpha_0}-2\varpi_{\alpha_1}+(a_1-1)\varpi_\beta-\sum\limits_{\substack{\chi_l\in S \\ \chi_l\neq\beta \\ \chi_l\neq\alpha_0 \\ \chi_l\neq\alpha_1}}\varpi_{\chi_l}$$ and that $G$ is the product of $G_0$ and a group of type $G_2$ whose simple roots are $\alpha_0$ and $\alpha_1$. The $G$-module $H^2(Y,N_{Y/\mathbb{X}})$ is not trivial if and only if there exists two distinct simple reflections $s$ and $s'$ such that $ss'(w_0^{\alpha_0,\beta}(\chi_Y)-\rho)$ is strictly antidominant. It is easy to see that both simple reflections have to be associated to simple roots of $G_0$. We distinguish several situations, by adopting a similar strategy as in Remark~\ref{rem:uniquecoeffpositif}, ie by restricting the cases that can get such $s$ and $s'$. 

If $s=s_{\chi_l}$ and $s'=s_{\chi_{l'}}$ are both different from $s_\beta$, then $s's(\chi_l^\vee)$ is a negative coroot with zero coefficient in $\beta^\vee$ and then $ss'(w_0^{\alpha_0,\beta}(\chi_Y)-\rho)$ is not antidominant. Suppose now that $\beta$ is either $\chi_l$ or $\chi_{l'}$.
Note also that, if $\chi_l$ and $\chi_{l'}$ are orthogonal, then $s's(\chi_l^\vee)=-\chi_l^\vee$ and $s's(\chi_{l'}^\vee)=-\chi_{l'}^\vee$ so that $ss'(w_0^{\alpha_0,\beta}(\chi_Y)-\rho)$ is not antidominant. Then we can also suppose that $\chi_l$ and $\chi_{l'}$ are linked in the Dynking diagram. 
Now, if $\beta=\chi_l$, since $s_\beta(\chi_l^\vee)$ is a positive coroot and $s's(s_\beta(\chi_l^\vee))=-\chi_l^\vee$, $ss'(w_0^{\alpha_0,\beta}(\chi_Y)-\rho)$ cannot be antidominant. 

To summarize from now, we proved that, $H^2(Y,N_{Y/\mathbb{X}})$ is not trivial if and only if there exists a simple root $\chi_l$ linked with $\beta$ such that $s_{\chi_l}s_\beta(w_0^{\alpha_0,\beta}(\chi_Y)-\rho)$ is strictly antidominant.\\

We compute that $s_\beta s_{\chi_l}(\beta^\vee)=(\langle \chi_l,\beta^\vee\rangle\langle \beta,\chi_l^\vee\rangle-1)\beta^\vee-\langle\chi_l,\beta^\vee\rangle\chi_l^\vee$, so that the coefficient of $s_{\chi_l}s_\beta(w_0^{\alpha_0,\beta}(\chi_Y)-\rho)$ in $\varpi_\beta$ is $(a_1-1)(\langle \chi_l,\beta^\vee\rangle\langle \beta,\chi_l^\vee\rangle-1)+\langle \chi_l,\beta^\vee\rangle$, which is negative (or equivalently $\leq -1$) if and only if either $\langle \chi_l,\beta^\vee\rangle\langle \beta,\chi_l^\vee\rangle=1$ (ie $\beta$ and $\chi_l$ linked by a simple arrow) or $$a_1\leq 1+\frac{-\langle \chi_l,\beta^\vee\rangle-1}{\langle \chi_l,\beta^\vee\rangle\langle \beta,\chi_l^\vee\rangle-1}=\left\lbrace\begin{array}{ccc}
   1  & \mbox{if} & -\langle \chi_l,\beta^\vee\rangle=1 \\
   2  & \mbox{if} & -\langle \chi_l,\beta^\vee\rangle=2\mbox{ or }3
\end{array} \right.$$
Moreover, $s_\beta s_{\chi_l}(\chi_l^\vee)=-\chi_l^\vee+\langle\beta,\chi_l^\vee\rangle\beta^\vee$. In particular, if $a_1=2$ and $-\langle\beta,\chi_l^\vee\rangle=1$ (which occurs when $-\langle \chi_l,\beta^\vee\rangle\geq 2$ or $\langle \chi_l,\beta^\vee\rangle\langle \beta,\chi_l^\vee\rangle=1$), the coefficient of $s_{\chi_l}s_\beta(w_0^{\alpha_0,\beta}(\chi_Y)-\rho)$ in $\varpi_{\chi_l}$ is 0. We conclude that the only case left is when $\langle \chi_l,\beta^\vee\rangle\langle \beta,\chi_l^\vee\rangle=1$ and $a_1\geq 3$. We then focus on that left case.

If there exists a second simple root $\delta$ linked to $\beta$, then $s_\beta s_{\chi_l}(\delta^\vee)=s_\beta(\delta^\vee)=\beta^\vee+\delta^\vee$ and the coefficient  of $s_{\chi_l}s_\beta(w_0^{\alpha_0,\beta}(\chi_Y)-\rho)$ in $\varpi_{\delta}$ is $a_1-2>0$. Similarly, if there exists a second simple root $\delta$ linked to $\chi_l$, then $s_\beta s_{\chi_l}(\delta^\vee)=s_\beta(\delta^\vee+\chi_l^\vee)=\beta^\vee+\delta^\vee+\chi_l^\vee$ and the coefficient  of $s_{\chi_l}s_\beta(w_0^{\alpha_0,\beta}(\chi_Y)-\rho)$ in $\varpi_{\delta}$ is $a_1-3\geq 0$. We conclude that, in case (XVIII) (8) with any $G_0$ and $a_1\geq 2$, $H^2(Y,N_{Y/\mathbb{X}})$ is not trivial if and only if $G_0$ is of type $A_3$ and $a_1\geq 3$. (And in that case, $H^2(Y,N_{Y/\mathbb{X}})$ is the $\operatorname{SL}_3\times G_2$-module $V((a_1-3)\varpi_1+\varpi_{\alpha_1})$, where $\varpi_1$ is the first fundamental weight of $\operatorname{SL}_3$.) This concludes the proof of Theorem~\ref{th:obstruction}.\\

Note that, in case (XVIII) (8) with $G_0$ of type $A_3$ and $a_3\geq 3$, we do not know if local deformations are obstructed or not.



\section{Some explicit deformations in few cases}\label{sec4}
Here, we give two cases where we can explicitly describe a family of deformations, by using the horospherical description of the varieties. Unfortunately, the construction does not generalize to all cases.

\subsection{$G = \operatorname{SL}_2\times \{1\} \times \Cbb^*$ and $a_1\geq 2$}\label{Hirzebruch}

In that situation, the varieties $\mathbb{X}$ are well-known varieties : indeed they are toric (under the action of $T\times  \{1\} \times \Cbb^*\subseteq \operatorname{SL}_2\times \{1\} \times \Cbb^*$), and isomorphic to the Hirzebruch surfaces $\mathbb{F}_{a_1}$. The deformations of  Hirzebruch surfaces are known and explicitly described in  \cite{Manetti04} : $\mathbb{F}_q$ deforms into $\mathbb{F}_{q-2k}$ for any $k$ such that $q-2k\geq 0$. Note that $\mathbb{F}_0=\Pbb^1\times\Pbb^1$ and $\mathbb{F}_1=\operatorname{Bl}_0\Pbb^2$ are locally rigid.\\

We can construct an explicit family of deformations of $\mathbb{F}_2$ to $\Pbb^1\times\Pbb^1$ as follows.

As a horospherical $\operatorname{SL}_2$-variety, $\mathbb{X}=\mathbb{F}_2$ is defined as the closure of the $\operatorname{SL}_2$-orbit of $[e_1+e_1\otimes e_1\otimes e_1]$ in $\Pbb(\Cbb^2\oplus S^3\Cbb^2)$, where $(e_1,e_2)$ is a basis of $\Cbb^2$. But the $\operatorname{SL}_2$-module $V(\varpi_1)\oplus V(3\varpi_1)=\Cbb^2\oplus S^3\Cbb^2$ is isomorphic to $\Cbb^2\otimes S^2\Cbb^2$. Indeed, in $\Cbb^2\otimes S^2\Cbb^2$,  a highest weight of weight  $3\varpi_1$ is clearly $e_1\otimes (e_1\otimes e_1)$, and a highest weight of weight  $\varpi_1$ is  $e_2\otimes (e_1\otimes e_1)-\frac{1}{2} e_1\otimes (e_1\otimes e_2+e_2\otimes e_1)$.

For any $\lambda\in\Cbb$, we consider $x_\lambda:=e_1+e_1\otimes e_1\otimes e_1+\lambda (e_2\otimes e_1\otimes e_1 + e_1\otimes e_2\otimes e_1 +e_1\otimes e_1\otimes e_2$ in $\Cbb^2\oplus S^3\Cbb^2$; and $\mathfrak{X}_\lambda$ the closure of the $\operatorname{SL}_2$-orbit of $[x_\lambda]$ in $\Pbb(\Cbb^2\oplus S^3\Cbb^2)$. 

For $\lambda=0$, it is clear that $\mathfrak{X}_0$ is the variety $\mathbb{X}$.
And for any $\lambda\neq 0$, by identifying $e_1$ with $2\lambda(e_2\otimes (e_1\otimes e_1)-\frac{1}{2} e_1\otimes (e_1\otimes e_2+e_2\otimes e_1))$ in the isomorphism $\Cbb^2\oplus S^3\Cbb^2\simeq\Cbb^2\otimes S^2\Cbb^2$, the vector $x_\lambda$ identifies to $(e_1+3\lambda e_2)\otimes (e_1\otimes e_1)$, and $\mathfrak{X}_{\lambda}$ is then isomorphic to the closure of the $\operatorname{SL}_2$-orbit of $([e_1+3\lambda e_2], [e_1])$ in $\Pbb^1\times\Pbb^1$ (embedded in $\Pbb^5$, with $\mathcal{O}(1)\times\mathcal{O}(2)$-polarization), which is $\Pbb^1\times\Pbb^1$.\\

In fact, this construction corresponds to the horospherical degeneration (\cite[section 2.2]{AlexeevBrion05} or \cite{Knop90}) of the spherical (but not horospherical) $\operatorname{SL}_2$-variety $\Pbb^1\times\Pbb^1$ (with diagonal action). We cannot expect the same with $a_1=3$ because $\operatorname{Bl}_0\Pbb^2$ is already a horospherical $\operatorname{SL}_2$-variety.

	\subsection{Case (I)(3) with $a_1=1$}
	
	The construction above also works in that case,  corresponding to the horospherical degeneration of the product $\mathbb{P}^m\times\mathbb{G}(i;m+1)$ (of a projective space and a grassmannian), under $\operatorname{SL}_{m+1}$-action.\\
	
	Let $G=\SL_{m+1}$, ie of type $A_m$ with $m\geq 3$. Let $2\leq i\leq m-1$ and set $\beta=\chi_1$, $\alpha_0=\chi_{i+1}$ and $\alpha_1=\chi_i$.
	Then $\mathbb{X}$ is the closure of $G\cdot [v_{\varpi_1+\varpi_{i+1}}+v_{2\varpi_1+\varpi_i}]$ in $\mathbb{P}(V(\varpi_1+\varpi_{i+1})\oplus V(2\varpi_1+\varpi_i))$.

	\begin{lem}\label{lem:submoduleI3}
	We have $V(\varpi_1+\varpi_{i+1})\oplus V(2\varpi_1+\varpi_i)\subseteq V(2\varpi_1)\otimes V(\varpi_i)$.
	\end{lem}
	
	\begin{proof}
	It is enough to find highest weight vectors (ie $B$-stable vectors) of weights $2\varpi_1+\varpi_i$ and $\varpi_1+\varpi_{i+1}$ in $V(2\varpi_1)\otimes V(\varpi_i)$. It is obvious that $v_{2\varpi_1}\otimes v_{\varpi_i}$ is a highest weight vector of weight $2\varpi_1+\varpi_i$.
	
	For the second weight, consider the vector $$v_0:=\sum_{j=1}^{i+1} (-1)^{j-1} (e_1\otimes e_j+e_j\otimes e_1)\otimes e_1\wedge\cdots\wedge e_{j-1}\wedge e_{j+1}\wedge\cdots\wedge e_{i+1}.$$
	
	We easily check that $T$ acts with weight $2\epsilon_1+\epsilon_2+\cdots+\epsilon_{i+1}=\varpi_1+\varpi_{i+1}$ on $v_0$. Let $1\leq h<k\leq m+1$. For any $\lambda\in \Cbb$, let $b_{hk}(\lambda)$ be the element of $G$ that sends $e_k$ to $e_k+\lambda e_h$ and fixes all others $e_j$'s.
	
	For any $1\leq j\leq i+1$, denote by $u_j$ the vector $(e_1\otimes e_j+e_j\otimes e_1)\otimes e_1\wedge\cdots\wedge e_{j-1}\wedge e_{j+1}\wedge\cdots\wedge e_{i+1}$. Then $$b_{hk}(\lambda)\cdot u_j= \left\lbrace\begin{array}{ccc}
	u_j & if & j\neq h,\,k\\
	u_j+\lambda(e_1\otimes e_h+e_h\otimes e_1)\otimes e_1\wedge\cdots\wedge e_{k-1}\wedge e_{k+1}\wedge\cdots\wedge e_{i+1} & if & j= k\\
	u_j-(-1)^{k-h}\lambda(e_1\otimes e_h+e_h\otimes e_1)\otimes e_1\wedge\cdots\wedge e_{k-1}\wedge e_{k+1}\wedge\cdots\wedge e_{i+1} & if & j= h
\end{array}\right.	 $$

Then, we can check that $b_{hk}(\lambda)\cdot v_0=v_0$. And, since $T$ and the $b_{hk}(\lambda)$ generate $B$, we have proved that $v_0$ is a highest weight vector of weight $\varpi_1+\varpi_{i+1}$.
	\end{proof}
	
	A consequence of the lemma is that, for any $\lambda\in\Cbb^*$, $\mathbb{X}$ can be seen as the closure of $G\cdot [\lambda v_0+v_{2\varpi_1}\otimes v_{\varpi_i}]$ in $\mathbb{P}(V(2\varpi_1)\otimes V(\varpi_i))$.
	
	Set $v_1:=\frac{u_1}{2}+(-1)^{i-1}u_{i+1}$ and, for any $2\leq j\leq i$, $v_j:=u_j+(-1)^{i-j}u_{i+1}$.
	
	\begin{lem}
	For any $1\leq j\leq i$, $v_j$ is in $V(2\varpi_1+\varpi_i)$.
	\end{lem}
	
	\begin{proof}
	Recall that $V(2\varpi_1+\varpi_i)$ is the $G$-module generated by $v_{2\varpi_1+\varpi_i}:=e_1\otimes e_1\otimes e_1\wedge\cdots\wedge e_i$. 
	
	Let us prove that $v_1$ is in $V(2\varpi_1+\varpi_i)$. By mapping $e_1$ to $e_1+e_{i+1}$ we transform $v_{2\varpi_1+\varpi_i}$ into $$\begin{array}{cccc}
	& e_1\otimes e_1\otimes e_1\wedge\cdots\wedge e_i&+& (-1)^{i-1}v_1 \\
	 +& (-1)^{i-1}(e_1\otimes e_{i+1}+e_{i+1}\otimes e_1\otimes e_2\wedge\cdots\wedge e_{i+1})& + & e_{i+1}\otimes e_{i+1}\otimes e_1\wedge\cdots\wedge e_i\\
	 +&  (-1)^{i-1}e_{i+1}\otimes e_{i+1}\otimes e_2\wedge\cdots\wedge e_{i+1}. & &\end{array}$$ And by mapping $e_1$ to $e_1-e_{i+1}$ we get $$\begin{array}{cccc}
	& e_1\otimes e_1\otimes e_1\wedge\cdots\wedge e_i&-& (-1)^{i-1}v_1 \\
	 +& (-1)^{i-1}(e_1\otimes e_{i+1}+e_{i+1}\otimes e_1\otimes e_2\wedge\cdots\wedge e_{i+1})& + & e_{i+1}\otimes e_{i+1}\otimes e_1\wedge\cdots\wedge e_i\\
	 -&  (-1)^{i-1}e_{i+1}\otimes e_{i+1}\otimes e_2\wedge\cdots\wedge e_{i+1}. & &\end{array}$$ Then the difference $2(-1)^{i-1}(v_1+e_{i+1}\otimes e_{i+1}\otimes e_2\wedge\cdots\wedge e_{i+1})$ of both vectors is in $V(2\varpi_1+\varpi_i)$. But $e_{i+1}\otimes e_{i+1}\otimes e_2\wedge\cdots\wedge e_{i+1}$ can be obtained from $e_1\otimes e_1\otimes e_1\wedge\cdots\wedge e_i$ by exchanging $e_1$ and $e_{i+1}$, so that it is in $V(2\varpi_1+\varpi_i)$ and we conclude.\\
	
	Note that, by exchanging $e_2$ with $e_j$ for $3\leq j\leq i$, we get that $v_2$ is in $V(2\varpi_1+\varpi_i)$ if and only if $v_j$ is. Then, it remains to prove that $v_2\in V(2\varpi_1+\varpi_i)$. First by mapping $e_1$ to $e_1\pm e_2$ we get $e_1\otimes e_1\otimes e_1\wedge\cdots\wedge e_i
	\pm (e_1\otimes e_2+e_2\otimes e_1)\otimes e_1\wedge\cdots\wedge e_i 
	+e_2\otimes e_2\otimes e_1\wedge\cdots\wedge e_i$.  We then deduce (as above) that $(e_1\otimes e_2+e_2\otimes e_1)\otimes e_1\wedge\cdots\wedge e_i$ is in $V(2\varpi_1+\varpi_i)$. Now, from the latter vector, by mapping  $e_2$ to $e_2\pm e_{i+1}$, we obtain $(e_1\otimes e_2+e_2\otimes e_1)\otimes e_1\wedge\cdots\wedge e_i\pm (-1)^{i-1}v_2 + (-1)^{i-2}(e_1\otimes e_{i+1}+e_{i+1}\otimes e_1)\otimes e_1\wedge e_3\cdots\wedge e_{i+1}$. This implies that $v_2$ is also in $V(2\varpi_1+\varpi_i)$
	\end{proof}

Notice that $$\begin{array}{ccc}v_0=\sum_{j=1}^{i+1}(-1)^{j-1}u_j& = &	u_1+\sum_{j=2}^{i}((-1)^{j-1}v_j-(-1)^{i-1}u_{i+1})+(-1)^iu_{i+1}\\
& = & u_1+\sum_{j=2}^{i}(-1)^{j-1}v_j +i(-1)^{i}u_{i+1}\\ 
& = & (1+\frac{i}{2})u_1-i v_1 +\sum_{j=2}^{i}(-1)^{j-1}v_j. 
\end{array} $$
In particular, $v_0+i v_1 -\sum_{j=2}^{i}(-1)^{j-1}v_j=(1+\frac{i}{2})u_1$.
	
	For any $\lambda\in\Cbb^*$, set $$z_\lambda:=\lambda v_0+ v_{2\varpi_1+\varpi_i}-\frac{\lambda i}{2}v_1-\lambda\sum_{j=2}^i (-1)^{j-1}v_j=e_1\otimes e_1\otimes (e_1+\lambda(i+2)e_{i+1})\wedge e_2\wedge\cdots\wedge e_i.$$
	
	\begin{lem}\label{lem:generalfiberI3} For any $\lambda\in\Cbb^*$, the closure $X_\lambda$ of $G\cdot [z_\lambda]$ in $\mathbb{P}(V(2\varpi_1)\otimes V(\varpi_i))$ is isomorphic to the product $\mathbb{P}^m\times\mathbb{G}(i;m+1)$ (of a projective space and a grassmannian). 
	\end{lem}
	
	\begin{proof}
	Note that $\mathbb{P}^m\times\mathbb{G}(i;m+1)$ can be embedded in $\mathbb{P}(V(2\varpi_1)\otimes V(\varpi_i))$ as the $G\times G$-orbit of $z_\lambda$. It implies that, for any $\lambda\in\Cbb^*$, $X_\lambda$ is isomorphic to a subvariety of $\mathbb{P}^m\times\mathbb{G}(i;m+1)$. But $G$ acts on $\mathbb{P}^m\times\mathbb{G}(i;m+1)$ with two orbits: a closed one consisting of couples $(l,V_i)$ such that $l\subset V_i$, and an open one consisting of couples $(l,V_i)$ such that $l\not\subset V_i$. We conclude by the fact that $z_\lambda$ is in the open orbit.
	\end{proof}
	
	\begin{prop}
	The variety $\mathbb{X}$ admits a deformation into $\mathbb{P}^m\times\mathbb{G}(i;m+1)$.
	\end{prop}
	
	\begin{proof}
	With the notation above, we consider the family  $\mathfrak{X}\longrightarrow \mathbb{A}^1$ defined by 
	
	$$\mathfrak{X}_\lambda:=\overline{G\cdot[v_{\varpi_1+\varpi_{i+1}}+ v_{2\varpi_1+\varpi_i}-\frac{\lambda i}{2}v_1-\lambda\sum_{j=2}^i (-1)^{j-1}v_j]}\subset \Pbb(V(\varpi_1+\varpi_{i+1})\oplus V(2\varpi_1+\varpi_i)).$$
	In particlular, $\mathfrak{X}_0\simeq \mathbb{X}$.
	
	Let $\lambda\neq 0$. 
By identifying $V(\varpi_1+\varpi_{i+1})\oplus V(2\varpi_1+\varpi_i)$ with a submodule of  $V(2\varpi_1)\otimes V(\varpi_i)$ and $v_{\varpi_1+\varpi_{i+1}}$ with $\lambda v_0$ as in Lemma~\ref{lem:submoduleI3}, we get that $\mathfrak{X}_\lambda$ is isomorphic to $\mathbb{P}^m\times\mathbb{G}(i;m+1)$ by Lemma~\ref{lem:generalfiberI3}.
	\end{proof}

	\appendix
	\section{Study of cases in order to compute $H^1(\mathbb{X},T_\mathbb{X})$}
	
	We explain how to study all cases with minimal computations. In each case (including prime cases), we should compute $A$, $B$, $c_{\alpha_i}$ and $d_{\alpha_i}$ for $i=1$ (when we consider the orbit $Y$) and for $i=0$ (when we consider the orbit $Z$). We begin with several usuful remarks.

\begin{enumerate}

\item The condition of Theorem \ref{theo1} or Theorem \ref{theo2} saying that $\beta$ cannot be linked with another simple root than $\alpha_0$ and $\alpha_1'$, or  $\alpha_0'$ and $\alpha_1$, permits to extract less cases. Note that, when we have a non prime (resp. prime) case where $\beta$ is only linked with $\alpha_0$ and $\alpha_1'$, then the corresponding prime (resp. non prime) case is such that  $\beta$ is only linked with $\alpha_0'$ and $\alpha_1$, and conversely. 

For example : (I) (3)
\begin{dynkinDiagram}A{*.**.**}
\node[below,/Dynkin diagram/text style] at (root 5)
{\(\beta\)};
\node[below,/Dynkin diagram/text style] at (root 2)
{\(\alpha_0\)};
\node[below,/Dynkin diagram/text style] at (root 3)
{\(\alpha_1\)};
\node[above,/Dynkin diagram/text style] at (root 4)
{\(\alpha_1'\)};
\end{dynkinDiagram} could give some $H^1$ for $Y$, and then (I)(3)'
\begin{dynkinDiagram}A{*.**.**}
\node[below,/Dynkin diagram/text style] at (root 5)
{\(\beta\)};
\node[below,/Dynkin diagram/text style] at (root 2)
{\(\alpha_1\)};
\node[below,/Dynkin diagram/text style] at (root 3)
{\(\alpha_0\)};
\node[above,/Dynkin diagram/text style] at (root 4)
{\(\alpha_0'\)};
\end{dynkinDiagram} could give some $H^1$ for $Z$.
We notice that $c_{\alpha_1}$ and $d_{\alpha_1}$ in the non prime (resp. prime) case have the same values as $c_{\alpha_0}$ and $d_{\alpha_0}$ in the prime (resp. non prime) case; as well as $A$ and $B$ which are not changed with the double exchange prime/non prime and $Y/Z$. \\

So we will list cases where $\beta$ is only linked with $\alpha_0$ and $\alpha_1'$, and compute (when necessary) $A=-\langle \beta,\alpha_0^\vee\rangle$, $B=-\langle \beta,\alpha_1'^\vee\rangle$, $c_{\alpha_1}$ and $d_{\alpha_1}$. Then we will be able to say if $H^1(Y,N_{Y/X})\neq 0$ with $A$, $B$, $c_{\alpha_1}$, $d_{\alpha_1}$ in that case, and also if $H^1(Z,N_{Z/X})\neq 0$ in the prime or non prime corresponding other case. Recall that $c_{\alpha_1}$ and $d_{\alpha_1}$ are the coefficients in $\alpha_1^\vee$ of the biggest coroots (in $R^\vee$) with coefficents 1 in $\alpha_0^\vee$ and 0 in $\beta^\vee$, resp. 0 in $\alpha_0^\vee$ and 1 in $\beta^\vee$, in $R^\vee$. \\

\item Recall that if $A=B=0$, then $G_0$ has to be of type $A_1$ : we are in cases (XVIII) (from (1) to (8), prime and non prime), or in cases $G=G_0\times G_1\times G_2$. In these cases, $\lambda$ could be any positive integer, so that $a_1$ could be any integer $\geq 2$ (because $d_{\alpha_1}=0$). And we always have $H^1(Z,N_{Z/X})= 0$ in the prime or non prime corresponding other case because $d_{\alpha_0}=d_{\alpha_1}=0<2$.\\

As soon as $A$ or $B$ is non zero, the conditions of the theorem imply that $A$ and $B$ are at most 1, and $\lambda=1$. In particular, $a_1=2-d_{\alpha_1}$ for $Y$ and $a_1=d_{\alpha_0}-2$ for $Z$. since $A$ and $B$ are $\leq 1$, we do not have the cases where $G_0$ is of type $B_m$ and $\beta=\chi_{m-1}$, $G_0$ is of type $C_m$ and $\beta=\chi_{m}$, $G_0$ is of type $F_4$ and $\beta=\chi_{2}$, and $G_0$ is of type $G_2$ and $\beta=\chi_{2}$.

Notice also that $c_{\alpha_1}$ has to be $\leq 1$ when we consider $Y$, with equality if and only if $A=0$.
 (and $c_{\alpha_0}$ has to be $\leq 1$ when we consider $Z$ with equality if and only if $A=0$). In particular, in cases (I) to (VIII), where $\alpha_0$ and $\alpha_1$ are in the same connected component of $\Gamma_{S\backslash\{\beta\}}$, we need to have $A=0$ so that $\beta$ has to be extremal (and $\alpha_1'$ linked to $\beta$, for $Y$, and $\alpha_0'$ linked to $\beta$, for $Z$). \\

For all cases left, we have to check the conditions of the theorem.

Note that we can first consider the orbit $Y$ and compute $A$, $B$, $c_{\alpha_1}$ and $d_{\alpha_1}$. If $d_{\alpha_1}\leq 1$ then we have $H^1(Y,N_{Y/X})\neq 0$, if $d_{\alpha_1}\geq 3$, then we have $H^1(Z,N_{Z/X})\neq 0$ for the other prime or non prime case. And if $d_{\alpha_1}=2$, we are in the case where $a_1=0$, prime and non prime cases are the same, we only consider non prime cases, and precise if $H^1(Y,N_{Y/X})\neq 0$ (if non prime) or $H^1(Z,N_{Z/X})\neq 0$ (if prime).

We deal with the cases where $\alpha_0$ or $\alpha_1$ are imaginary, as if they were an extremal simple root of the groups $G_1$ or $G_2$ of types $A$ (and the other $\alpha$ is not in the same simple group). Indeed, the conventions $A=0$ and $c_{\alpha_1}=0$ if $\alpha_0$ imaginary, and $B=0$ and $c_{\alpha_1}=d_{\alpha_1}=0$ if $\alpha_1$ imaginary, give the same results as if, $\alpha_0$ or $\alpha_1$ respectively, is an extremal simple root of the groups $G_1$ or $G_2$ of types $A$.

\end{enumerate}

Let us give now the complete study of cases, by following all remarks above.\\

1.  List of cases, by using conditions :
\begin{enumerate}[(a)]
\item $\beta$ is linked at most with $\alpha_0$ and $\alpha_1'$ in the Dynkin diagram;
\item $A$ and $B$ are at most 1;
\item in cases  (I) to (VIII), $\beta$ is extremal and $\alpha_1'$ is linked to $\beta$;
\end{enumerate}
\begin{multicols}{2}
(I) $G_0$ of type $A_m$, $m\geq 3$ \\

(3) \begin{dynkinDiagram}A{*.**.**}
\node[below,/Dynkin diagram/text style] at (root 5)
{\(\beta\)};
\node[below,/Dynkin diagram/text style] at (root 2)
{\(\alpha_0\)};
\node[below,/Dynkin diagram/text style] at (root 3)
{\(\alpha_1\)};
\node[above,/Dynkin diagram/text style] at (root 4)
{\(\alpha_1'\)};
\end{dynkinDiagram}
~\\

(II) $G_0$ of type $B_m$, $m\geq 3$ \\

(3) \begin{dynkinDiagram}B{*.**.**}
\node[below,/Dynkin diagram/text style] at (root 5)
{\(\beta\)};
\node[below,/Dynkin diagram/text style] at (root 2)
{\(\alpha_0\)};
\node[below,/Dynkin diagram/text style] at (root 3)
{\(\alpha_1\)};
\node[above,/Dynkin diagram/text style] at (root 4)
{\(\alpha_1'\)};
\end{dynkinDiagram}
~\\

(5)' with $k=1$ \begin{dynkinDiagram}B{**.**}
\node[below,/Dynkin diagram/text style] at (root 1)
{\(\beta\)};
\node[below,/Dynkin diagram/text style] at (root 4)
{\(\alpha_0\)};
\node[below,/Dynkin diagram/text style] at (root 3)
{\(\alpha_1\)};
\node[above,/Dynkin diagram/text style] at (root 2)
{\(\alpha_1'\)};
\end{dynkinDiagram}
~\\

(III) $G_0$ of type $C_m$, $m\geq 3$ \\

(5)' with $k=1$ \begin{dynkinDiagram}C{**.**.**}
\node[below,/Dynkin diagram/text style] at (root 1)
{\(\beta\)};
\node[below,/Dynkin diagram/text style] at (root 4)
{\(\alpha_0\)};
\node[below,/Dynkin diagram/text style] at (root 3)
{\(\alpha_1\)};
\node[above,/Dynkin diagram/text style] at (root 2)
{\(\alpha_1'\)};
\end{dynkinDiagram}
~\\

(IV) $G_0$ of type $D_m$, $m\geq 4$ \\

(2)' \begin{dynkinDiagram}D{*.**.***}
\node[right,/Dynkin diagram/text style] at (root 6)
{\(\beta\)};
\node[right,/Dynkin diagram/text style] at (root 5)
{\(\alpha_0\)};
\node[below,/Dynkin diagram/text style] at (root 1)
{\(\alpha_1\)};
\node[right,/Dynkin diagram/text style] at (root 4)
{\(\alpha_1'\)};
\end{dynkinDiagram}
~\\

(5) with $k=1$
\begin{dynkinDiagram}D{**..***}
\node[below,/Dynkin diagram/text style] at (root 1)
{\(\beta\)};
\node[right,/Dynkin diagram/text style] at (root 5)
{\(\alpha_0\)};
\node[right,/Dynkin diagram/text style] at (root 4)
{\(\alpha_1\)};
\node[above,/Dynkin diagram/text style] at (root 2)
{\(\alpha_1'\)};
\end{dynkinDiagram}
~\\

(V) $G_0$ of type $E_6$ No case\\

(VI) $G_0$ of type $E_7$ No case\\

(VII) $G_0$ of type $E_8$ No case\\

(VIII) $G_0$ of type $F_4$ \\

 (2)' \begin{dynkinDiagram}F4
\node[below,/Dynkin diagram/text style] at (root 1)
{\(\beta\)};
\node[below,/Dynkin diagram/text style] at (root 3)
{\(\alpha_0\)};
\node[below,/Dynkin diagram/text style] at (root 2)
{\(\alpha_1\)};
\node[above,/Dynkin diagram/text style] at (root 2)
{\(\alpha_1'\)};
\end{dynkinDiagram}
~\\
 
(5)  \begin{dynkinDiagram}F4
\node[below,/Dynkin diagram/text style] at (root 4)
{\(\beta\)};
\node[below,/Dynkin diagram/text style] at (root 2)
{\(\alpha_0\)};
\node[below,/Dynkin diagram/text style] at (root 3)
{\(\alpha_1\)};
\node[above,/Dynkin diagram/text style] at (root 3)
{\(\alpha_1'\)};
\end{dynkinDiagram}
~\\

(6) \begin{dynkinDiagram}F4
\node[below,/Dynkin diagram/text style] at (root 4)
{\(\beta\)};
\node[below,/Dynkin diagram/text style] at (root 1)
{\(\alpha_0\)};
\node[below,/Dynkin diagram/text style] at (root 3)
{\(\alpha_1\)};
\node[above,/Dynkin diagram/text style] at (root 3)
{\(\alpha_1'\)};
\end{dynkinDiagram}
~\\

(IX) $G_0$ of type $A_m$, $m\geq 2$ \\

(1) \begin{dynkinDiagram}A{**.**}\node[below,/Dynkin diagram/text style] at (root 1)
{\(\beta\)};
\node[below,/Dynkin diagram/text style] at (root 2)
{\(\alpha_0\)};
\end{dynkinDiagram}
\begin{dynkinDiagram}A{**.**}
\node[below,/Dynkin diagram/text style] at (root 1)
{\(\alpha_1\)};
\end{dynkinDiagram} or $\alpha_1$ is imaginary.
~\\

(2)' \begin{dynkinDiagram}A{**.**}\node[below,/Dynkin diagram/text style] at (root 1)
{\(\beta\)};
\node[below,/Dynkin diagram/text style] at (root 4)
{\(\alpha_1\)};
\node[above,/Dynkin diagram/text style] at (root 2)
{\(\alpha_1'\)};
\end{dynkinDiagram}
\begin{dynkinDiagram}A{**.**}
\node[below,/Dynkin diagram/text style] at (root 1)
{\(\alpha_0\)};
\end{dynkinDiagram} or $\alpha_0$ is imaginary.
~\\

(5)', (7) with $k=2$, (12), (17)' \begin{dynkinDiagram}A{**.***.**}
\node[below,/Dynkin diagram/text style] at (root 4)
{\(\beta\)};
\node[below,/Dynkin diagram/text style] at (root 5)
{\(\alpha_0\)};
\node[below,/Dynkin diagram/text style] at (root 1)
{\(\alpha_1\)};
\node[above,/Dynkin diagram/text style] at (root 3)
{\(\alpha_1'\)};
\end{dynkinDiagram}
~\\

(X) $G_0$ of type $B_m$, $m\geq 2$ \\

(15)' \begin{dynkinDiagram}B{**.**}
\node[below,/Dynkin diagram/text style] at (root 4)
{\(\beta\)};
\node[below,/Dynkin diagram/text style] at (root 1)
{\(\alpha_1\)};
\node[above,/Dynkin diagram/text style] at (root 3)
{\(\alpha_1'\)};
\end{dynkinDiagram}
\begin{dynkinDiagram}A{**.**}
\node[below,/Dynkin diagram/text style] at (root 1)
{\(\alpha_0\)};
\end{dynkinDiagram} or $\alpha_0$ is imaginary.
~\\

(16) \begin{dynkinDiagram}B{**.**}
\node[below,/Dynkin diagram/text style] at (root 4)
{\(\beta\)};
\node[below,/Dynkin diagram/text style] at (root 3)
{\(\alpha_0\)};
\end{dynkinDiagram}
\begin{dynkinDiagram}A{**.**}
\node[below,/Dynkin diagram/text style] at (root 1)
{\(\alpha_1\)};
\end{dynkinDiagram} or $\alpha_1$ is imaginary.
~\\

(XI) $G_0$ of type $C_m$, $m\geq 2$ \\

(1) \begin{dynkinDiagram}C{**.**}
\node[below,/Dynkin diagram/text style] at (root 4)
{\(\beta\)};
\node[below,/Dynkin diagram/text style] at (root 3)
{\(\alpha_0\)};
\end{dynkinDiagram}
\begin{dynkinDiagram}A{**.**}
\node[below,/Dynkin diagram/text style] at (root 1)
{\(\alpha_1\)};
\end{dynkinDiagram} or $\alpha_1$ is imaginary.
~\\

(4)' \begin{dynkinDiagram}C{**.***.**}
\node[below,/Dynkin diagram/text style] at (root 4)
{\(\beta\)};
\node[below,/Dynkin diagram/text style] at (root 5)
{\(\alpha_0\)};
\node[below,/Dynkin diagram/text style] at (root 1)
{\(\alpha_1\)};
\node[above,/Dynkin diagram/text style] at (root 3)
{\(\alpha_1'\)};
\end{dynkinDiagram}
~\\

(7) with $k=m-1$ (and (4) with $m=3$) \begin{dynkinDiagram}C{**.***}
\node[below,/Dynkin diagram/text style] at (root 4)
{\(\beta\)};
\node[below,/Dynkin diagram/text style] at (root 3)
{\(\alpha_0\)};
\node[below,/Dynkin diagram/text style] at (root 5)
{\(\alpha_1\)};
\node[above,/Dynkin diagram/text style] at (root 5)
{\(\alpha_1'\)};
\end{dynkinDiagram}
~\\

(XII) $G_0$ of type $D_m$ No case\\

(XIII) $G_0$ of type $E_6$ No case\\

(XIV) $G_0$ of type $E_7$ No case\\

(XV) $G_0$ of type $E_8$ No case \\

(XVI)\\

(9)'
\begin{dynkinDiagram}F4
\node[below,/Dynkin diagram/text style] at (root 3)
{\(\beta\)};
\node[below,/Dynkin diagram/text style] at (root 4)
{\(\alpha_0\)};
\node[below,/Dynkin diagram/text style] at (root 1)
{\(\alpha_1\)};
\node[above,/Dynkin diagram/text style] at (root 2)
{\(\alpha_1'\)};
\end{dynkinDiagram}
~\\

(11) \begin{dynkinDiagram}F4
\node[below,/Dynkin diagram/text style] at (root 3)
{\(\beta\)};
\node[below,/Dynkin diagram/text style] at (root 2)
{\(\alpha_0\)};
\node[below,/Dynkin diagram/text style] at (root 4)
{\(\alpha_1\)};
\node[above,/Dynkin diagram/text style] at (root 4)
{\(\alpha_1'\)};
\end{dynkinDiagram}
~\\

(XVII)\\

(1) \begin{dynkinDiagram}[backwards = true]G2
\node[below,/Dynkin diagram/text style] at (root 2)
{\(\beta\)};
\node[below,/Dynkin diagram/text style] at (root 1)
{\(\alpha_0\)};
\end{dynkinDiagram}
\begin{dynkinDiagram}A{**.**}
\node[below,/Dynkin diagram/text style] at (root 1)
{\(\alpha_1\)};
\end{dynkinDiagram} or $\alpha_1$ is imaginary.
~\\

(1)' \begin{dynkinDiagram}[backwards = true]G2
\node[below,/Dynkin diagram/text style] at (root 2)
{\(\beta\)};
\node[below,/Dynkin diagram/text style] at (root 1)
{\(\alpha_1\)};
\end{dynkinDiagram}
\begin{dynkinDiagram}A{**.**}
\node[below,/Dynkin diagram/text style] at (root 1)
{\(\alpha_0\)};
\end{dynkinDiagram} or $\alpha_0$ is imaginary\\

\end{multicols}

~\\
This ends all cases where $A$ or $B$ is non zero.
We will consider the cases where $A=B=0$ at the end.\\

2. We give all needed values in the following table and give the value of $a_1$ such that $H^1(Y,N_{Y/X})\neq 0$ or $H^1(ZY,N_{Z/X})\neq 0$. We keep an empty box if there is no $a_1$ such that $H^1(Y,N_{Y/X})\neq 0$ or $H^1(Z,N_{Z/X})\neq 0$ (when $c_{\alpha_1}\geq 2$).

Note that the latter case comes from the computation in the case (XVII) (1)', and since $d_{\alpha_1}>2$ it gives a non zero $H^1$ for the orbit $Z$ in the non prime case.\\

\begin{tabular}{|c|c|c|c|c|c|}
\hline
Case & $A$ & $B$ & $c_{\alpha_1}$ & $d_{\alpha_1}$ & $a_1$ (for $\lambda=1$)\\

\hline

(I) (3) & 0 & 1 & 1 & 1 & 1 \\

\hline

(II) (3) & 0 & 1 & 1 & 2 & 0 for $Y$ \\

\hline

(II)  (5)' & 0 & 1 & 2 & 1 & \\

\hline

(III) (5)' & 0 & 1 & 1 & 1 & 1 \\

\hline

(IV) (2)' & 0 & 1 & 1 & 1 & 1\\

\hline

(IV) (5) with $k=1$ & 0 & 1 & 1 & 1 & 1 \\

\hline

(VIII) (2)'  & 0 & 1 & 2 & & \\

\hline

(VIII) (5)  & 0 & 1 & 1 & 1 & 1\\

\hline

(VIII) (6)  & 0 & 1 & 1 & 2 & 0 for $Y$\\

\hline

\end{tabular}

\begin{tabular}{|c|c|c|c|c|c|}
\hline
Case & $A$ & $B$ & $c_{\alpha_1}$ & $d_{\alpha_1}$ & $a_1$ (for $\lambda=1$)\\

\hline

(IX) (1)  & 1 & 0 & 0 & 0 & 2\\

\hline

(IX) (2)'  & 0 & 1  & 0 & 1 & 1\\

\hline

(IX) (5)'  & 1 &1   & 0 & 1 & 1\\

\hline

(IX) (7) with $k=2$ & 1 & 1  & 0 & 1 & 1\\

\hline

(IX) (12)  & 1 & 1  & 0 & 1 & 1\\

\hline

(IX) (17)'  & 1 & 1  & 0 & 1 & 1\\

\hline

(X) (15)  & 0 & 1  & 0 & 2 & 0 for $Z$\\

\hline

(X) (16)  & 1 &  0 & 0 & 0 & 2\\

\hline

(XI) (1)  & 1 &  0 & 0 & 0 & 2\\

\hline

(XI) (4)'  & 1 & 1  & 0 & 1 & 1\\

\hline

(XI) (4) with $m=3$  & 1 & 1 & 0 & 2 & 0 for $Y$\\

\hline

(XI) (7) with $k=m-1$  & 1 &  1 & 0 & 2 & 0 for $Y$\\

\hline

(XVI) (9) & 1 & 1 & 0 & 2 & 0  for $Z$\\

\hline

(XVI) (11) & 1 & 1 & 0 & 1 & 1\\

\hline

(XVII) (1) & 1 & 0 & 0 & 0 & 2 \\

\hline

(XVII) (1) & 0 & 1 & 0 & 3 & 1 for $Z$ \\

\hline 

\end{tabular}

~\\

Now we deal with the cases (XVIII) from (1) to (7), prime and non prime, with $G_0$ of type $A_1$ (and $\beta=\chi_1$). Notice that, in all cases, $d_{\alpha_1}=0$ (and $d_{\alpha_0}=0$). So we only have to check if $c_{\alpha_1}$ is less than 1 (and if $c_{\alpha_0}$ is less than 1 for the prime case). We notice that in types $A$ and $D$, the cases prime and non prime are the same by symmetry of the Dynkin Diagram; and in types $B$, $C$ and $F_4$, only the non prime case gives some non zero $H^1(Y,N_{Y/X})$ because, for prime cases we have $c_{\alpha_1}\geq 2$. Here are the cases where $H^1(Y,N_{Y/X})\neq 0$, for any $a_1\geq 2$.

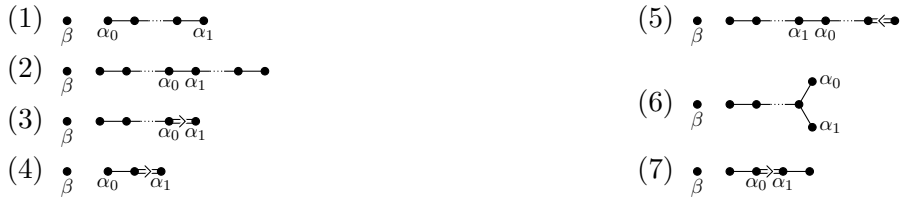
\begin{multicols}{2}
(1) \begin{dynkinDiagram}A1
\node[below,/Dynkin diagram/text style] at (root 1)
{\(\beta\)};
\end{dynkinDiagram}
\begin{dynkinDiagram}A{**.**}
\node[below,/Dynkin diagram/text style] at (root 1)
{\(\alpha_0\)};
\node[below,/Dynkin diagram/text style] at (root 4)
{\(\alpha_1\)};
\end{dynkinDiagram}

(2) \begin{dynkinDiagram}A1
\node[below,/Dynkin diagram/text style] at (root 1)
{\(\beta\)};
\end{dynkinDiagram}
\begin{dynkinDiagram}A{**.**.**}
\node[below,/Dynkin diagram/text style] at (root 3)
{\(\alpha_0\)};
\node[below,/Dynkin diagram/text style] at (root 4)
{\(\alpha_1\)};
\end{dynkinDiagram}

(3) \begin{dynkinDiagram}A1
\node[below,/Dynkin diagram/text style] at (root 1)
{\(\beta\)};
\end{dynkinDiagram}
\begin{dynkinDiagram}B{**.**}
\node[below,/Dynkin diagram/text style] at (root 3)
{\(\alpha_0\)};
\node[below,/Dynkin diagram/text style] at (root 4)
{\(\alpha_1\)};

\end{dynkinDiagram}

(4) \begin{dynkinDiagram}A1
\node[below,/Dynkin diagram/text style] at (root 1)
{\(\beta\)};
\end{dynkinDiagram}
\begin{dynkinDiagram}B3
\node[below,/Dynkin diagram/text style] at (root 1)
{\(\alpha_0\)};
\node[below,/Dynkin diagram/text style] at (root 3)
{\(\alpha_1\)};
\end{dynkinDiagram}

(5) \begin{dynkinDiagram}A1
\node[below,/Dynkin diagram/text style] at (root 1)
{\(\beta\)};
\end{dynkinDiagram}
\begin{dynkinDiagram}C{**.**.**}
\node[below,/Dynkin diagram/text style] at (root 4)
{\(\alpha_0\)};
\node[below,/Dynkin diagram/text style] at (root 3)
{\(\alpha_1\)};
\end{dynkinDiagram}

(6) \begin{dynkinDiagram}A1
\node[below,/Dynkin diagram/text style] at (root 1)
{\(\beta\)};
\end{dynkinDiagram}
\begin{dynkinDiagram}D{**.***}
\node[right,/Dynkin diagram/text style] at (root 4)
{\(\alpha_0\)};
\node[right,/Dynkin diagram/text style] at (root 5)
{\(\alpha_1\)};
\end{dynkinDiagram}

(7) \begin{dynkinDiagram}A1
\node[below,/Dynkin diagram/text style] at (root 1)
{\(\beta\)};
\end{dynkinDiagram}
\begin{dynkinDiagram}F4
\node[below,/Dynkin diagram/text style] at (root 2)
{\(\alpha_0\)};
\node[below,/Dynkin diagram/text style] at (root 3)
{\(\alpha_1\)};
\end{dynkinDiagram}
\end{multicols}

Note that case (XVIII) (8) is directly detailed in Section~\ref{sec2}.\\

The only case left is when $G=G_0\times G_1\times G_2$, $G_0$ is of type $A_1$ and $a_1\geq 2$ (including the cases where $\alpha_0$ or/and $\alpha_1$ are imaginary): 
\begin{dynkinDiagram}[backwards = true]A1
\node[below,/Dynkin diagram/text style] at (root 1)
{\(\beta\)};
\end{dynkinDiagram}
\begin{dynkinDiagram}A{**.**}
\node[below,/Dynkin diagram/text style] at (root 1)
{\(\alpha_0\)};
\end{dynkinDiagram}
\begin{dynkinDiagram}A{**.**}
\node[below,/Dynkin diagram/text style] at (root 1)
{\(\alpha_1\)};
\end{dynkinDiagram} \\

\newpage
\bibliographystyle{amsalpha}
\bibliography{biblio.bib}

\providecommand{\bysame}{\leavevmode\hbox to3em{\hrulefill}\thinspace}
\providecommand{\MR}{\relax\ifhmode\unskip\space\fi MR }
\providecommand{\MRhref}[2]{%
  \href{http://www.ams.org/mathscinet-getitem?mr=#1}{#2}
}
\providecommand{\href}[2]{#2}
\begin{thebibliography}{Man04}

\bibitem[AB05]{AlexeevBrion05}
Valery Alexeev and Michel Brion, \emph{Moduli of affine schemes with reductive
  group action}, J. Algebr. Geom. \textbf{14} (2005), no.~1, 83--117 (English).

\bibitem[Bou07]{Bourbaki2007}
Nicolas Bourbaki, \emph{{Groupes et alg{\`e}bres de Lie}}, Springer, Berlin,
  Heidelberg, 2007.

\bibitem[Bri89]{briondiv}
Michel Brion, \emph{Groupe de picard et nombres caract\'eristiques des
  vari\'et\'es sph\'eriques}, Duke Math. J. \textbf{58} (1989), no.~2,
  397--424.

\bibitem[Bri97]{brioncurves}
M.~Brion, \emph{Curves and divisors in spherical varieties}, Algebraic groups
  and {L}ie groups, Austral. Math. Soc. Lect. Ser., vol.~9, Cambridge Univ.
  Press, Cambridge, 1997, pp.~21--34.

\bibitem[Kno90]{Knop90}
Friedrich Knop, \emph{Weylgruppe und {Momentabbildung}. ({Weyl} group and
  moment map)}, Invent. Math. \textbf{99} (1990), no.~1, 1--23 (German).

\bibitem[Man04]{Manetti04}
Marco Manetti, \emph{Lectures on deformations of complex manifolds.
  {Deformations} from differential graded viewpoint}, Rend. Mat. Appl., VII.
  Ser. \textbf{24} (2004), no.~1, 1--183 (English).

\bibitem[Pas08]{Pasquier2006}
Boris Pasquier, \emph{Vari\'et\'es horosph\'eriques de {Fano}}, Bulletin de la
  Soci\'et\'e Math\'ematique de France \textbf{136} (2008), no.~2, 195--225.

\bibitem[Pas09]{Pasquier2009}
\bysame, \emph{On some smooth projective two-orbit varieties with {{Picard}}
  number 1}, Mathematische Annalen \textbf{344} (2009), no.~4, 963--987.

\bibitem[Pas20]{Pasquier2020}
\bysame, \emph{Smooth projective horospherical varieties of {{Picard}} group
  {{Z}}{$^{2}$}}, {\'E}pijournal de G{\'e}om{\'e}trie Alg{\'e}brique
  \textbf{Volume 4} (2020).

\bibitem[Per18]{Perrin2018}
Nicolas Perrin, \emph{Sanya lectures: geometry of spherical varieties}, Acta
  Math. Sin., Engl. Ser. \textbf{34} (2018), no.~3, 371--416 (English).

\bibitem[PP10]{Pasquier2010}
Boris Pasquier and Nicolas Perrin, \emph{Local rigidity of quasi-regular
  varieties}, Mathematische Zeitschrift \textbf{265} (2010), no.~3, 589--600.

\bibitem[San13]{Sano2013}
Taro Sano, \emph{Unobstructedness of deformations of weak fano manifolds},
  International Mathematics Research Notices \textbf{2014} (2013).

\bibitem[Vil]{Villeneuve2025}
L\'ea Villeneuve, \emph{Geometry of horospherical varieties of rank 1, smooth,
  projective, with {Picard} number 2}, preprint available at
  https://hal.science/hal-05400358.

\end{thebibliography}

\end{document}